\documentclass[a4paper, 12pt]{article}

\usepackage[top = 2.5cm, bottom = 2.5cm, left = 2.4cm, right = 2.4cm]{geometry}
\usepackage{amsfonts, amsmath, amsthm, amssymb, mathtools, cases, bbm, array}
\usepackage{lmodern, indentfirst, setspace, changepage, fancyhdr}
\usepackage{tikz, pgfplots, float, subcaption}
\usepackage[title]{appendix}
\usepackage[english]{babel}
\usepackage{cite}

\usepackage{hyperref}
\hypersetup{
	colorlinks = true,
	linkcolor = blue,
	citecolor = blue,
}

\newtheorem{proposition}{Proposition}
\newtheorem{corollary}{Corollary}
\newtheorem{theorem}{Theorem}
\newtheorem{lemma}{Lemma}
\theoremstyle{definition}
\newtheorem{assumption}{Assumption}
\newtheorem{condition}{Condition}
\newtheorem{definition}{Definition}

\newtheorem{remark}{Remark}

\newcommand{\ind}[1]{\mathbbm{1}_{\left\{#1\right\}}}
\newcommand{\norm}[1]{\left|\left|#1\right|\right|}
\newcommand{\floor}[1]{\left\lfloor#1\right\rfloor}

\newcommand{\map}[3]{#1 : #2 \longrightarrow #3}
\newcommand{\set}[2]{\left\{#1 : #2\right\}}

\newcommand{\vct}[2]{\left(#1 : #2\right)}

\newcommand{\defeq}{\vcentcolon=}
\newcommand{\eqdef}{=\vcentcolon}

\newcommand{\bw}{\boldsymbol{w}}
\newcommand{\bx}{\boldsymbol{x}}
\newcommand{\by}{\boldsymbol{y}}
\newcommand{\bz}{\boldsymbol{z}}
\newcommand{\bA}{\boldsymbol{A}}

\newcommand{\bS}{\boldsymbol{S}}

\newcommand{\calK}{\mathcal{K}}
\newcommand{\calL}{\mathcal{L}}

\newcommand{\scdot}{{}\cdot{}}

\newcommand{\N}{\mathbbm{N}}

\newcommand{\R}{\mathbbm{R}}
\newcommand{\Z}{\mathbbm{Z}}
\newcommand{\e}{\mathrm{e}}

\newcommand{\condp}{P\probarg}
\newcommand{\conde}{E\expectarg}
\newcommand{\condr}[1]{#1\probarg}
\newcommand{\conds}[1]{#1\expectarg}
\DeclarePairedDelimiterX{\probarg}[1]{(}{)}{%
	\ifnum\currentgrouptype=16 \else\begingroup\fi
	\activatebar#1
	\ifnum\currentgrouptype=16 \else\endgroup\fi
}
\DeclarePairedDelimiterX{\expectarg}[1]{[}{]}{%
	\ifnum\currentgrouptype=16 \else\begingroup\fi
	\activatebar#1
	\ifnum\currentgrouptype=16 \else\endgroup\fi
}
\newcommand{\innermid}{\nonscript\;\delimsize\vert\nonscript\;}
\newcommand{\activatebar}{%
	\begingroup\lccode`\~=`\|
	\lowercase{\endgroup\let~}\innermid 
	\mathcode`|=\string"8000
}

\usetikzlibrary{automata, arrows, positioning, calc, external, babel, backgrounds, matrix, shapes}
\usepgfplotslibrary{fillbetween}
\pgfplotsset{
	compat = 1.16,
	ticklabel style = {font = \footnotesize},
	every axis/.append style = {
		grid style = {dashed, gray, opacity = 0.2},
		label style = {font = \footnotesize}, 
		width = \columnwidth,
		height = 0.618 * 1 * \columnwidth
	}
}

\definecolor{britishracinggreen}{rgb}{0.0, 0.26, 0.15}
\definecolor{bostonuniversityred}{rgb}{0.8, 0.0, 0.0}
\definecolor{ceruleanblue}{rgb}{0.16, 0.32, 0.75}
\definecolor{airforceblue}{rgb}{0.36, 0.54, 0.66}
\definecolor{cadmiumgreen}{rgb}{0.0, 0.42, 0.24}
\definecolor{ao(english)}{rgb}{0.0, 0.5, 0.0}
\definecolor{coolblack}{rgb}{0.0, 0.18, 0.39}
\definecolor{byzantine}{rgb}{0.74, 0.2, 0.64}
\definecolor{alizarin}{rgb}{0.82, 0.1, 0.26}
\definecolor{arsenic}{rgb}{0.23, 0.27, 0.29}
\definecolor{cobalt}{rgb}{0.0, 0.28, 0.67}
\definecolor{amber}{rgb}{1.0, 0.75, 0.0}

\title{Asymptotically Optimal Policies for\\Weakly Coupled Markov Decision Processes \vspace{\baselineskip}}

\author{
\begin{tabular}{ccc}
	\normalsize Diego Goldsztajn & \hspace{1cm} & \normalsize Konstantin Avrachenkov \\
	\footnotesize Universidad ORT Uruguay & \hspace{1cm} & \footnotesize Inria \\
	\footnotesize Montevideo, Uruguay & \hspace{1cm} & \footnotesize Sophia Antipolis, France \\
	\scriptsize\texttt{goldsztajn@ort.edu.uy} & \hspace{1cm} &\scriptsize\texttt{konstantin.avrachenkov@inria.fr} \\
\end{tabular}
}

\date{\vspace{\baselineskip} January 23, 2026}

\begin{document}

	
\maketitle

\noindent\rule{\textwidth}{1pt}

\vspace{0.3\baselineskip}

\onehalfspacing

\begin{adjustwidth}{0.8cm}{0.8cm}
	\begin{center}
		\textbf{Abstract}
	\end{center}
	
	\vspace{0.2\baselineskip}
	
	\noindent We consider the problem of maximizing the expected average reward obtained over an infinite time horizon by $n$ weakly coupled Markov decision processes. Our setup is a substantial generalization of the multi-armed restless bandit problem that allows for multiple actions and constraints. We establish a connection with a deterministic and continuous-variable control problem where the objective is to maximize the average reward derived from an occupancy measure that represents the empirical distribution of the processes when $n \to \infty$. We show that a solution of this fluid problem can be used to construct policies for the weakly coupled processes that achieve the maximum expected average reward as $n \to \infty$, and we give sufficient conditions for the existence of solutions. Under certain assumptions on the constraints, we prove that these conditions are automatically satisfied if the unconstrained single-process problem admits a suitable unichain and aperiodic policy. In particular, the assumptions include multi-armed restless bandits and a broad class of problems with multiple actions and inequality constraints. Also, the policies can be constructed in an explicit way in these cases. Our theoretical results are complemented by several concrete examples and numerical experiments, which include multichain setups that are covered by the theoretical results.
	
	\vspace{0.2\baselineskip}
	
	
	\vspace{0.2\baselineskip}
	
	\small\noindent Most of this work was carried out while Diego Goldsztajn was with Inria. This research was partially supported by ANII through the fellowship  PD\_NAC\_2024\_1\_182118.
	
	\vspace{0.2\baselineskip}
	
	\small\noindent This paper has been accepted at the Journal of Applied Probability.  
\end{adjustwidth}

\newpage

\section{Introduction}
\label{sec: introduction}

Consider $n$ identical discrete-time Markov decision processes with finite state and action spaces. Given the states and actions for all the processes, the states reached at the next time step are mutually independent. Nevertheless, the processes are not independent since there exist linear constraints that couple the selection of actions across the processes. The objective is to maximize the expected average reward obtained over an infinite time horizon by selecting the actions in a suitable way that respects the constraints.

The latter setup can be used to model a wide range of applications, including the healthcare, queueing, supply chain and scheduling problems considered in \cite{gocgun2012lagrangian,hawkins2003langrangian,meuleau1998solving,patrick2008dynamic,salemi2018approximate}. The processes typically represent independent agents, projects or tasks which use shared resources of different types. Selecting an action for a process determines the amount of resources of each type that the process will use over the current time step, as well as the average reward that the process will produce. The constraints can be used to impose conditions on the total amount of each resource that should be used at each time step; e.g., the total amount of resource used may not exceed the amount that is available.

The set of weakly coupled Markov decision processes can be regarded as a single Markov decision process without any constraints, where the constraints are implicitly imposed in the definition of the action space associated with each state. Thus, an optimal policy for the weakly coupled processes can be obtained by leveraging dynamic programming techniques. However, the optimal policy hardly ever admits a closed form expression, and the time required to compute the policy numerically can often be exceedingly large as it depends exponentially on the number of processes $n$. This creates a strong incentive for studying simple policies that may not be optimal but become optimal as $n \to \infty$.

While weakly coupled processes have been considered before, the focus has been on problems with finite time horizons or multi-armed restless bandits where only two actions are possible and there is a single constraint with a specific structure. In particular, we are not aware of any previous results for problems with multiple actions and constraints where the goal is to maximize the expected average reward over an infinite time horizon. Our results concern this class of problems and hold under weak assumptions on the transition probability matrices of the processes. These assumptions are provably weaker than those in all previous papers on multi-armed restless bandits, perhaps with the exception of \cite{hong2023restless}. Although we do not prove that the assumptions in \cite{hong2023restless} are stronger, we provide examples where our approach yields asymptotically optimal policies and that in \cite{hong2023restless} does not.

\subsection{Overview of the main results}
\label{sub: overview of the main results}

Using standard arguments, we establish that a linear program yields an upper bound for the expected average reward that any policy can provide. Then we prove that a solution of a deterministic and continuous-variable control problem can be used to construct policies for the weakly coupled processes that achieve this upper bound in the limit as $n \to \infty$. We further analyze this fluid problem and derive sufficient conditions for the existence of solutions, as well as structural properties that help to construct a solution.

Under certain assumptions on the constraints of the problem, we show that the latter conditions are automatically satisfied if the unconstrained single-process problem admits a suitable unichain\footnote{A Markov chain is unichain if its state space can be partitioned into a single recurrent class and a possibly empty set of transient states. A policy is unichain if the associated Markov chain is unichain, and a Markov decision process is unichain if every stationary deterministic policy is unichain.} and aperiodic policy; such a policy may exist even if the single-process problem is multichain. The assumptions include multi-armed restless bandit problems and a broad class of problems with multiple actions and inequality constraints. In both of these cases we provide a simple procedure for constructing asymptotically optimal policies, which we apply to concrete examples and evaluate numerically.

\subsubsection{Mean-field limit and asymptotic optimality}
\label{subsub: asymptotic optimality results}

An important characteristic of the weakly coupled processes considered here is that processes in the same state react to the same action in an identically distributed way. In view of this property, we describe the states of all the processes through a occupancy measure or state frequency vector that reflects the empirical distribution of the states. Moreover, we focus our search for asymptotically optimal policies on the class of policies which are defined by mappings that assign to each state frequency vector a state-action frequency vector, i.e., a vector that for each state and action indicates the fraction of processes that are in the given state and for which the given action is selected.

Such a mapping determines a transition kernel that governs the evolution of the state frequency vector as a discrete-time Markov chain. Assuming that the mapping converges uniformly as $n \to \infty$ to a continuous function that we call fluid control, we show that the latter stochastic process converges weakly to a nearly deterministic process dubbed fluid trajectory. More precisely, a fluid trajectory has a possibly random initial condition but evolves over time in a deterministic manner, as a discrete-time dynamical system that is governed by the fluid control. The values taken by a fluid trajectory are occupancy measures in a probability simplex and applying the fluid control to any such measure yields a distribution of state-action pairs which satisfies the constraints of the problem, because the fluid control is the limit of mappings which satisfy the constraints.

The latter mean-field limit connects the weakly coupled Markov decision processes with the above-mentioned fluid problem. Specifically, the fluid problem is to find a continuous fluid control such that the average reward of any fluid trajectory converges to the upper bound provided by the linear program. Assuming that such a fluid control exists, we use an interchange of limits argument to prove that asymptotically optimal policies can be obtained by constructing mappings that converge uniformly to the fluid control. 

It follows from the latter result that the problem of finding asymptotically optimal policies can be decomposed into the following subproblems. First, find a fluid control that solves the fluid problem. Second, construct a discrete approximation of the fluid control for each $n$. In the case of multi-armed restless bandit problems or the class of problems with multiple actions and inequality constraints mentioned earlier, we establish that the second step can be carried out following a straightforward rounding procedure.

\subsubsection{Structure of optimal fluid controls}
\label{subsub: structure of fluid controls}

In order to find suitable fluid controls, we fix an occupancy measure that solves the linear program and search for continuous fluid controls such that all fluid trajectories converge to this optimal occupancy measure. For this restricted version of the fluid problem, we derive a necessary and sufficient condition for the existence of solutions. The condition involves the set of occupancy measures that assign zero mass to some state in the support of the optimal occupancy measure; this set is contained in the boundary of the simplex. We prove that a solution of the restricted fluid problem exists if and only if there exists a fluid control such that no fluid trajectory is contained in the latter set. Furthermore, we give an explicit expression for a solution using the latter auxiliary fluid control.

Loosely speaking, we solve the restricted fluid problem by expressing every occupancy measure as the sum of two nonnegative measures, such that the first measure is a multiple of the optimal occupancy measure with maximal total mass. The solution of the restricted fluid problem is the fluid control that at each time step keeps the first measure unchanged and applies the auxiliary fluid control to the second measure. The property of the auxiliary fluid control ensures that the first measure absorbs mass from the second measure over time. As a result, all fluid trajectories converge to the optimal occupancy measure.

The latter global attractivity result reduces the problem of finding an optimal fluid control to that of finding an auxiliary fluid control. For problems with general constraints, we consider auxiliary fluid controls that can be expressed as a convex combination of two other fluid controls. The first one is based on a policy for the unconstrained single-process problem, whereas the purpose of the second one is to ensure that the convex combination complies with the constraints. We prove that the convex combination yields an auxiliary fluid control, with the desired property, if the policy for the unconstrained single-process problem is unichain, aperiodic and such that the unique irreducible class contains all the states that are in the support of the optimal occupancy measure.

It is difficult to provide an explicit expression for the auxiliary fluid control without imposing some assumptions on the structure of the constraints, because the fluid control must comply with the constraints. We obtain explicit expressions in two important cases: multi-armed restless bandits and a class of problems with multiple actions and inequality constraints. The latter inequality constraints must have nonnegative coefficients, which can be interpreted as the amounts of resources that are consumed by the different actions. In addition, there must exist an action that does not consume any resources. Applying such an action to a process does not contribute to violating any of the resource consumption constraints, which makes it easy to obtain feasible policies.


\subsection{Related work}
\label{sub: related work}

As noted earlier, the study of weakly coupled Markov decision processes with multiple actions and constraints has focused on scenarios with finite time horizons. In this context, several papers have studied linear programming and Lagrange decomposition techniques for obtaining approximate solutions and bounds; e.g., as in \cite{adelman2008relaxations,gocgun2012lagrangian,hawkins2003langrangian}. Recently, \cite{gast2024reoptimization,brown2023fluid,brown2022dynamic} have obtained asymptotic optimality results and performance bounds that depend on the number of processes and the time horizon; here the formulation in \cite{brown2022dynamic} allows for exogenous signals that affect the constraints, rewards and transition probabilities. Similar asymptotic optimality results and performance bounds have been established in the more restrictive scenario where there is a single constraint. In this setting, \cite{brown2020index,zhang2021restless,gast2023linear} have considered problems with two actions, whereas \cite{zayas2019asymptotically,xiong2021reinforcement} have studied problems with more than two actions. It is worth noting that the setups considered in \cite{brown2020index,brown2022dynamic,brown2023fluid} allow for heterogeneous processes, as well as time-dependent constraints, rewards and transition probabilities.

The problems studied in the above papers are substantially different from the ones considered here, where rewards are averaged over an infinite time horizon. In particular, the initial state and the length of the time horizon affect the optimal value and policy of a problem with a finite time horizon. In contrast, the expected average reward over an infinite time horizon is often not affected by the initial state, and the length of the time horizon is meaningless. Further, the structure of the transition kernels plays an important role in problems with the expected average reward criterion but is not particularly relevant when the time horizon is finite. Similar remarks apply to problems with an infinite time horizon and the discounted reward criterion, which can be regarded as problems with a finite but random time horizon. Such problems have been analyzed in \cite{brown2023fluid,brown2020index} by truncating the time horizon and applying results for problems with a finite time horizon.

To the best of our knowledge, the expected average reward criterion has only been considered when the processes are coupled by a single constraint. The focus has been on the multi-armed restless bandits introduced by Whittle, which have two actions and a single constraint that fixes the numbers of processes or arms that must select each action. In the seminal paper \cite{whittle1988restless}, Whittle defined an index policy for multi-armed restless bandits where a technical condition known as indexability holds, and conjectured the asymptotic optimality of this policy. A counterexample for this conjecture was provided in \cite{weber1990index}. However, \cite{weber1990index} also established that the conjecture is true when a dynamical system defined by the index policy has a global attractor; \cite{weber1990index} also assumed the ergodicity of the single-arm problem. Exponentially small bounds for the optimality gap were provided in \cite{gast2023exponential} when the same assumptions hold and the global attractor satisfies a nonsingularity condition. In addition, \cite{hodge2015asymptotic} generalized the asymptotic optimality result to multi-armed restless bandit problems with multiple actions and a single constraint with a specific structure.

The Whittle index policy is contained in the class of LP-priority policies introduced in \cite{verloop2016asymptotically}. The latter paper shows that an LP-priority policy is asymptotically optimal if the problem is unichain and a global attractor assumption holds, while \cite{gast2023linear} proves performance bounds. LP-priority policies can be more generally applied than the Whittle index policy since they do not depend on the indexability condition. However, both policies rely on global attractivity assumptions that must be checked for each problem instance; while some of our results also rely on a global attractor property, we formally prove this property instead of assuming that it holds as in \cite{weber1990index,verloop2016asymptotically,gast2023exponential,gast2023linear}. More recently, \cite{hong2023restless,hong2024unichain} have proposed novel policies for multi-armed restless bandits and have established asymptotic optimality results and performance bounds without using the global attractivity assumption.

As observed earlier, our assumptions are provably weaker than those in all of the above papers with the possible exception of \cite{hong2023restless}. Furthermore, we provide examples where our approach yields asymptotically optimal policies while the above-mentioned policies are not asymptotically optimal, including the one proposed in \cite{hong2023restless}. We also recall that our approach allows for multiple actions and constraints, while multi-armed restless bandits have only two actions and a single constraint with a specific structure.

Some of the above papers on multi-armed restless bandits have obtained bounds for the optimality gap as a function of the number of arms $n$. Moreover, some papers have derived bounds that decrease exponentially fast with $n$ by imposing stronger assumptions. Such bounds were first derived in \cite{gast2023exponential,gast2023linear} for the Whittle index and LP-priority policies by requiring, apart from the conditions mentioned earlier, that the solution of a linear program is nondegenerate in a suitable sense. Recently, \cite{hong2024exponential} has proposed a \emph{two-set policy} with the same performance guarantees under weaker assumptions: the uniform global attractivity property required in \cite{gast2023linear,gast2023exponential} is replaced by a local stability property which is further shown to be necessary for a broad class of problems. The focus of the current paper is not on bounding the optimality gap as a function of the number of processes $n$, but on obtaining simple policies that are asymptotically optimal under very general assumptions.

At approximately the same time as the conclusion of this work, similar results were derived independently in \cite{yan2024optimal} for the special case of multi-armed restless bandits. The latter paper also uses a deterministic control problem to construct policies that are asymptotically optimal, and the \emph{align and steer} policy considered in \cite{yan2024optimal} belongs to the class of policies analyzed in the present paper. However, the proof techniques are significantly different. The approach of \cite{yan2024optimal} is motivated by the certainty equivalence control principle of control theory and the proofs are based on the so-called Stein's method. In contrast, our approach is motivated by mean-field arguments and our proof techniques combine mean-field limits with interchange of limits and global attractivity arguments. Our results are more general and our approach has the advantage of establishing an explicit connection between the weakly coupled processes and the fluid problem, through mean-field limits.


\subsection{Organization of the paper}
\label{sub: organization of the paper}

The rest of the paper is organized as follows. In Section \ref{sec: problem formulation} we define the weakly coupled Markov decision processes formally and we discuss standard properties of optimal policies and a linear program relaxation. In Section \ref{sec: fluid heuristics} we prove the mean-field limit and asymptotic optimality results assuming that a solution of the fluid problem is available. In Section~\ref{sec: fluid controls} we establish sufficient conditions for the existence of solutions and we obtain structural properties that help to construct a solution. In Section \ref{sec: examples} we give concrete examples and evaluate the performance of our policies through numerical experiments. We also compare our policies with policies proposed in the literature on multi-armed restless bandits. In Section \ref{sec: conclusion} we provide some concluding remarks. Some results are proved in Appendix \ref{app: proofs of various results} and some details about the numerical examples are given in Appendix \ref{app: details about the examples}.

\section{Problem formulation}
\label{sec: problem formulation}

We consider the problem of maximizing the expected average reward obtained over an infinite time horizon by $n$ identical discrete-time Markov decision processes. Each process has the same finte action space $A$ and state space $S$; without loss of generality, we assume that in each state the controller has the same set of actions. Given the states and actions for all the processes, the states reached at the next time step are mutually independent, and the probability that a process moves to state $j$ given that it is in state $i$ and action $a$ is selected for it is denoted by $\condr{p}*{j | i, a}$. However, the processes are not independent since there exist linear constraints that couple the selection of actions across the processes.

The action and state of process $m \in \{1, \dots, n\}$ at time $t \in \N$ are the random variables denoted by $\bA_n(t, m)$ and $\bS_n(t, m)$, respectively; throughout the paper we use boldface notation for quantities that depend on time. We further define
\begin{equation}
	\label{eq: state and state-action frequencies}
	\bx_n(t, i) \defeq \frac{1}{n}\sum_{m = 1}^n \ind{\bS_n(t, m) = i} \quad \text{and} \quad \by_n(t, i, a) \defeq \frac{1}{n}\sum_{m = 1}^n \ind{\bS_n(t, m) = i, \bA_n(t, m) = a}
\end{equation}
for all $t \in \N$, $i \in S$ and $a \in A$. For each time $t$, the fraction of processes in state $i$ is given by the former quantity, while the latter quantity represents the fraction of processes that are in state $i$ for which action $a$ is selected. In particular, $\bx_n(t) \in \R^S$ and $\by_n(t) \in \R^{S \times A}$ are vectors describing the state and state-action frequencies at time $t$, respectively.

Consider matrices $\set{C_n(a)}{a \in A}$ and $\set{E_n(a)}{a \in A}$ with rows indexed by $i \in S$, such that the former matrices have $p$ columns and the latter matrices have $q$ columns. Further, let $d_n$ and $f_n$ be row vectors with $p$ and $q$ entries, respectively. The processes are weakly coupled by the following $p$ equality constraints and $q$ inequality constraints:
\begin{subequations}
	\label{eq: constraints}
	\begin{align}
		&\sum_{a \in A} \by_n(t, a) C_n(a) = d_n \quad \text{for all} \quad t \in \N, \label{seq: equality constraints} \\
		&\sum_{a \in A} \by_n(t, a) E_n(a) \preceq f_n \quad \text{for all} \quad t \in \N. \label{seq: inequality constraints}
	\end{align}
\end{subequations}
In the above expressions $\by_n(t, a) \defeq \vct{\by_n(t, i, a)}{i \in S}$ represents a row vector, and the notation $\preceq$ in the second line refers to the componentwise inequality.

At each time, a reward $r(i, a)$ is obtained by each process in state $i$ for which action $a$ is selected. Informally speaking, we must select the actions over time to maximize
\begin{equation}
	\label{eq: expected average reward}
	\lim_{T \to \infty} \frac{1}{T}\sum_{t = 0}^{T - 1} \frac{1}{n}\sum_{m = 1}^n \conde{r\left(\bS_n(t, m), \bA_n(t, m)\right)} = \lim_{T \to \infty} \frac{1}{T}\sum_{t = 0}^{T - 1} \sum_{a \in A} \conde{\by_n(t, a)}r(a),
\end{equation}
while complying with the constraints \eqref{eq: constraints}; the objective will be formally stated in the next section, where the existence of the limits is addressed. The expression on the left is called gain and represents the expected reward per unit of time obtained by one process, averaged across all the processes and over time. On the right, $r(a) \defeq \vct{r(i, a)}{i \in S}$ is a column vector, and the equality follows from the definition of the state-action frequencies.

For example, the processes could represent projects that are carried out in parallel and produce revenues or expenses over time; the latter correspond to negative rewards. The goal is to provide a policy for selecting the actions over time in a way that maximizes the average profit obtained per unit of time in the long-run. If the matrices and vectors in \eqref{eq: constraints} are nonnegative, then the relations in \eqref{eq: constraints} can be interpreted as constraints in the allocation of $p + q$ resources. Namely, performing action $a$ when a project is in state $i$ involves $C_n(i, k, a)$ units of resource $k \in \{1, \dots, p\}$ and $E_n(i, l, a)$ units of resource $p + l \in \{p + 1, \dots, p + q\}$. For instance, suppose that these resources are employees and trucks, respectively. Then there are $n d_n(k)$ employees and $n f_n(l)$ trucks in total. Also, \eqref{seq: equality constraints} says that all the employees must work in some project and \eqref{seq: inequality constraints} says that it is not necessary to use all the trucks.

If $A = \{0, 1\}$ and one of the following conditions holds, then the set of weakly coupled Markov decision processes is called multi-armed restless bandit.
\begin{enumerate}
	\item[(a)] There is a unique equality constraint and there are no inequality constraints. Also, the unique equality constraint satisfies that
	\begin{equation*}
		n d_n \in \{1, \dots, n\} \quad \text{while} \quad C_n(i, 0) = 0 \quad \text{and} \quad C_n(i, 1) = 1 \quad \text{for all} \quad i \in S.
	\end{equation*}
	
	\item[(b)] There is a unique inequality constraint and there are no equality constraints. Further, the unique inequality constraint satisfies that
	\begin{equation*}
		f_n \in [0, 1] \quad \text{while} \quad E_n(i, 0) = 0 \quad \text{and} \quad E_n(i, 1) = 1 \quad \text{for all} \quad i \in S.
	\end{equation*}
\end{enumerate}
In the above expressions $d_n$ and $f_n$ are scalars, and $C_n(a)$ and $E_n(a)$ are column vectors for each $a \in A$. As noted in Section \ref{sec: introduction}, multi-armed restless bandit problems have received considerable attention, under either of the latter conditions. The setup consider in the present paper is significantly more general.

\subsection{Existence of optimal policies}
\label{sub: existence of optimal policies}

The weakly coupled processes can be regarded as a single Markov decision process with finite state and action spaces. Specifically, the state space is $S^n$ and the action space for a given state is the subset of $A^n$ where the coupling constraints \eqref{eq: constraints} hold. More precisely, a state $(i_1, \dots, i_n) \in S^n$ and a tuple $(a_1, \dots, a_n) \in A^n$ uniquely determine a vector $y$ of state-action frequencies as in \eqref{eq: state and state-action frequencies}. The action space at a given state $(i_1, \dots, i_n)$ is the subset of $A^n$ such that the state-action frequency vector $y$ satisfies \eqref{eq: constraints}.

Consider the sets $\Pi_n^{SD} \subset \Pi_n^{SR} \subset \Pi_n^{HR}$ of stationary and deterministic, stationary and randomized and history-dependent and randomized policies for the latter Markov decision process; see \cite[Section 2.1.5]{puterman2014markov} for precise definitions. Loosely speaking, stationary policies select actions using solely the current state, whereas the history-dependent policies select actions using the current state and the history of all past states and actions. Randomized policies can be regarded as mappings that assign to each state, or history of states and actions, a probability distribution defined over the action space of the current state; an action is then sampled from this distribution. The policy is called deterministic if the latter probability distribution is always concentrated at a single action.

It follows from \cite[Proposition 8.1.1]{puterman2014markov} that the gain
\begin{equation*}
g_n^\pi(s) \defeq \lim_{T \to \infty} \frac{1}{T}\sum_{t = 0}^{T - 1} \frac{1}{n}\sum_{m = 1}^n \conds{E_s^\pi}*{r\left(\bS_n(t, m), \bA_n(t, m)\right)} = \lim_{T \to \infty} \frac{1}{T}\sum_{t = 0}^{T - 1} \sum_{a \in A} \conds{E_s^\pi}*{\by_n(t, a)}r(a)
\end{equation*}
exists for all policies $\pi \in \Pi_n^{SR}$ and initial states $s \in S^n$, i.e., the limits in \eqref{eq: expected average reward} exist. Here the expectation is taken with respect to the Markov chain defined by the initial state and the stationary and randomized policy. Also, \cite[Theorems 9.1.7 and 9.1.8]{puterman2014markov} imply that
\begin{equation*}
g_n^*(s) \defeq \sup_{\pi \in \Pi_n^{SD}} g_n^\pi(s)
\end{equation*}
is attained simultaneously for all $s \in S^n$ by some $\pi^* \in \Pi_n^{SD}$, and \cite[Theorem 9.1.6]{puterman2014markov} yields
\begin{equation*}
g_n^*(s) = \sup_{\pi \in \Pi_n^{HR}} \limsup_{T \to \infty} \frac{1}{T}\sum_{t = 0}^{T - 1} \frac{1}{n}\sum_{m = 1}^n \conds{E_s^\pi}*{r\left(\bS_n(t, m), \bA_n(t, m)\right)}.
\end{equation*}

In other words, there exists a stationary and deterministic policy $\pi^*$ that has maximal gain across all the initial states. This policy maximizes the expressions in \eqref{eq: expected average reward} over the set of policies for which the limits exist and provides an expected average reward not smaller than any history-dependent and randomized policy, as stated in the above equation.

An optimal policy $\pi^*$ does not typically admit a closed form expression. While $\pi^*$ can be computed numerically using dynamic programming, the computation time is exponential in the number of processes $n$ and thus exceedingly large in many applications. This creates an incentive for simple policies that perform well when $n$ is large.

\subsection{Fluid relaxation}
\label{sub: fluid relaxation}

In this paper we analyze a class of policies that are easy to compute and asymptotically optimal, i.e., they achieve the maximum gain as $n \to \infty$. For this purpose we adopt the following assumption throughout the paper. It says that \eqref{eq: constraints} is consistent in the limit.

\begin{assumption}
	\label{ass: limiting constraints}
	For each $a \in A$, the following limits exist and are finite:
	\begin{equation*}
		C(a) \defeq \lim_{n \to \infty} C_n(a), \quad d \defeq \lim_{n \to \infty} d_n, \quad E(a) \defeq \lim_{n \to \infty} E_n(a) \quad \text{and} \quad f \defeq \lim_{n \to \infty} f_n.
	\end{equation*}
\end{assumption}

Let $e \in \R^S$ denote the column vector with each of its entries equal to one. The processes of state and state-action frequencies $\bx_n$ and $\by_n$ take values in
\begin{equation*}
	X \defeq \set{x \in [0, 1]^{S}}{x e = 1} \quad \text{and} \quad Y \defeq \set{y \in [0, 1]^{S \times A}}{\sum_{a \in A} y(a)e = 1},
\end{equation*}
respectively; $x$ and $y(a) = \vct{y(i, a)}{i \in S}$ are regarded as row vectors. For each $a \in A$, let $P(a)$ denote the matrix such that the entry $(i, j)$ is $\condr{p}*{j | i, a}$. The following linear program will play an important role throughout the paper.

\begin{definition}
	\label{def: fluid relaxation}
	We call \emph{fluid relaxation} the following optimization problem
	\begin{equation}
		\label{pr: fluid relaxation}
		\begin{split}
			\underset{y \in Y}{\text{maximize}} &\quad \sum_{a \in A} y(a)r(a) \\
			\text{subject to} &\quad \sum_{a \in A} y(a)P(a) = \sum_{a \in A} y(a), \\
			&\quad \sum_{a \in A} y(a)C(a) = d, \\
			&\quad \sum_{a \in A} y(a)E(a) \preceq f.
		\end{split}
	\end{equation}
	Its optimal value is denoted by $g_r$.
\end{definition}

The optimization variables $y(i, a)$ can be interpreted as
\begin{equation*}
	\lim_{n \to \infty} \lim_{T \to \infty} \frac{1}{T}\sum_{t = 0}^{T - 1} \conde*{\by_n(t, i, a)}
\end{equation*}
when the limits exist. The first constraint in \eqref{pr: fluid relaxation} is a balance equation that is satisfied by long-run averages as defined by the inner limit; see the proof of the next proposition. The other constraints say that these long-run averages must satisfy \eqref{eq: constraints} as $n \to \infty$. Moreover, the objective of the problem can be interpreted as the limiting gain.

The following proposition says that $g_r$ is an upper bound for the gain of the weakly coupled Markov decision processes as $n \to \infty$. An intuitive explanation of this result, based on the above interpretation of \eqref{pr: fluid relaxation}, is that the constraints \eqref{eq: constraints} must hold at each time and for each sample path, whereas the constraints of \eqref{pr: fluid relaxation} are only imposed to long-run averages, thereby relaxing the problem. The proof of the proposition is rather standard but we provide it in Appendix \ref{app: proofs of various results} for completeness.

\begin{proposition}
	\label{prop: asymptotic upper bound for the gain}
	For each $a \in A$ and some fixed $n$, suppose that
	\begin{equation}
		\label{eq: constant constraints}
		C_n(a) = C(a), \quad d_n = d, \quad E_n(a) = E(a) \quad \text{and} \quad f_n = f.
	\end{equation}
	Then $g_n^*(s) \leq g_r$ for all $s \in S^n$. In the general case where \eqref{eq: constant constraints} may not hold, we have
	\begin{equation}
		\label{eq: asymptotic upper bound}
		\limsup_{n \to \infty} \sup_{s \in S^n} g_n^*(s) \leq g_r.
	\end{equation}
\end{proposition}

In the following two sections we will establish that the inequality in \eqref{eq: asymptotic upper bound} is in fact an equality under mild assumptions. Moreover, we will provide stationary policies for the $n$ weakly coupled processes with gains that approach $g_r$ as $n \to \infty$.

\section{Asymptotic optimality}
\label{sec: fluid heuristics}

An important characteristic of the weakly coupled processes is that two processes in the same state will react to the same action in an identically distributed way. Therefore, it makes sense to focus the search for asymptotically optimal policies on a class of policies that do not distinguish between two processes that are in the same state. More precisely, we may restrict our attention to stationary policies which can be defined by mappings that assign to each state frequency vector a state-action frequency vector. Mean-field theory suggests that then the state and state-action frequency processes $\bx_n$ and $\by_n$ behave deterministically in the limit as $n \to \infty$, which motivates the next definition.

\begin{definition}
	\label{def: fluid control}
	Consider a function $\map{\phi}{X}{Y}$ and let us use the short-hand notation $\phi(x, a)$ for $y(a)$ when $y = \phi(x)$. The latter function is called a fluid control if
	\begin{equation*}
		\sum_{a \in A} \phi(x, a) = x, \quad \sum_{a \in A} \phi(x, a)C(a) = d \quad \text{and} \quad \sum_{a \in A} \phi(x, a)E(a) \preceq f \quad \text{for all} \quad x \in X.
	\end{equation*}
	The fluid trajectory induced by a fluid control $\phi$ and an initial condition $x^0 \in X$ is the sequence $\set{\bx(t), \by(t)}{t \in \N}$ defined recursively by:
	\begin{equation*}
		\bx(0) = x^0, \quad \by(t) = \phi\left(\bx(t)\right) \quad \text{and} \quad \bx(t + 1) = \sum_{a \in A} \by(t, a)P(a) \quad \text{for all} \quad t \geq 0.
	\end{equation*}
\end{definition}

The deterministic vectors $\bx(t)$ and $\by(t)$ can be regarded as the fluid counterparts of $\bx_n(t)$ and $\by_n(t)$, respectively, while $\phi$ is the counterpart of a policy that assigns to each state frequency vector $x$ a state-action frequency vector $\phi(x)$. The first condition imposed on $\phi$ is that the state-action frequencies must be consistent with the state frequencies, whereas the other two conditions enforce the constraints \eqref{eq: constraints} with the matrices and vectors therein replaced by the limits in Assumption \ref{ass: limiting constraints}. The dynamics defining a fluid trajectory state that the fluid control $\phi$ determines the state-action frequencies $\by(t)$ as a function of the state frequencies $\bx(t)$, and then $\bx(t + 1)$ is obtained analogously to
\begin{equation*}
	\conde*{\bx_n(t + 1)} = \sum_{a \in A} \conde*{\by_n(t, a)}P(a).
\end{equation*}
This equation holds for the $n$ processes and any given policy, provided that the expectations are conditional on the policy and the initial states of all the processes.

Note that $\bx_n$ and $\by_n$ take values in the discrete spaces
\begin{equation*}
	X_n \defeq \set{x \in X}{nx \in \Z^S} \quad \text{and} \quad Y_n \defeq \set{y \in Y}{ny \in \Z^{S \times A}},
\end{equation*}
respectively. In general, a fluid control may not define a policy for the $n$ weakly coupled processes because its values may be outside of $Y_n$ or not satisfy \eqref{eq: constraints}. However, a principled way of obtaining policies that do not distinguish between processes in the same state is to first search for a fluid control $\map{\phi}{X}{Y}$ and then construct functions $\map{\phi_n}{X_n}{Y_n}$ that resemble $\phi$ and define admissible policies for the $n$ processes.

\begin{definition}
	\label{def: fluid heuristic}
	Consider functions $\map{\phi_n}{X_n}{Y_n}$ and let us use the short-hand notation $\phi_n(x, a)$ for $y(a)$ when $\phi_n(x) = y$, as for fluid controls. Also, assume that
	\begin{equation*}
		\sum_{a \in A} \phi_n(x, a) = x, \quad \sum_{a \in A} \phi_n(x, a)C_n(a) = d_n \quad \text{and} \quad \sum_{a \in A} \phi_n(x, a)E_n(a) \preceq f_n \quad \text{if} \quad x \in X_n.
	\end{equation*}
	The functions $\set{\phi_n}{n \geq 1}$ are discrete controls associated with a fluid control $\phi$ if
	\begin{equation*}
		\lim_{n \to \infty} \max_{x \in X_n} \norm{\phi(x) - \phi_n(x)} = 0;
	\end{equation*}
	the norm $\norm{\scdot}$ can be arbitrary.
\end{definition}

The mapping $\phi_n$ assigns to each $x \in X_n$ a state-action frequency vector $\phi_n(x) \in Y_n$ that satisfies the constraints \eqref{eq: constraints}. We say that a policy $\pi \in \Pi_n^{SR}$ realizes $\phi_n$ if
\begin{equation*}
	\by_n(t) = \phi_n\left(\bx_n(t)\right) \quad \text{for all} \quad t \geq 0
\end{equation*}
when $\pi$ is used. A policy that realizes $\phi_n$ can be obtained by computing $\by_n(t)$ as above and then selecting the $n \by_n(t, i, a)$ processes that use action $a$ in state $i$ uniformly at random among all the processes in state $i$. While a policy that realizes $\phi_n$ is generally not unique, the gain and the joint law of the processes $\bx_n$ and $\by_n$ are the same for any such policy; they are determined by $\phi_n$ and the initial state frequencies. Hence, we will focus on the mappings $\phi_n$ without specifying the policies that realize these mappings explicitly.

Later on in this section, we will assume that a fluid control $\phi$ is available, and that $\phi$ is a continuous function such that every fluid trajectory satisfies
\begin{equation}
	\label{eq: sufficient asymptotic optimality condition for fluid control}
	\lim_{t \to \infty} \sum_{a \in A} \by(t, a) r(a) = g_r,
\end{equation}
regardless of the initial condition. Using mean-field arguments, we will prove that the gain of any discrete control $\phi_n$ associated with $\phi$ approaches $g_r$ as $n \to \infty$; we will not require any conditions on the discrete controls besides those already stated in Definition \ref{def: fluid heuristic}.

In view of the latter result, the problem of finding asymptotically optimal policies can be decomposed in the following two subproblems.
\begin{enumerate}
	\item[(a)] \emph{Fluid problem:} construct a continuous fluid control $\phi$ that satisfies \eqref{eq: sufficient asymptotic optimality condition for fluid control}.
	
	\item[(b)] \emph{Approximation problem:} construct discrete controls $\phi_n$ as in Definition \ref{def: fluid heuristic}.
\end{enumerate}
Section \ref{sec: fluid controls} will focus on (a), which is the hardest. Then the discrete controls in (b) can be defined by letting $\phi_n(x)$ be the solution of the following optimization problem:
\begin{equation*}
	\begin{split}
		\underset{y \in Y_n}{\text{minimize}} &\quad \norm{y - \phi(x)} \\
		\text{subject to} &\quad \sum_{a \in A} y(a) = x, \\
		&\quad \sum_{a \in A} y(a)C_n(a) = d_n, \\
		&\quad \sum_{a \in A} y(a)E_n(a) \preceq f_n.
	\end{split}
\end{equation*}
However, simpler rounding procedures are possible under certain conditions; two such procedures and the corresponding conditions are discussed in Section \ref{sec: fluid controls}. A final step would be to define policies $\pi_n$ that realize the discrete controls $\phi_n$, but this is straightforward; e.g., it can be done through the random sampling procedure mentioned earlier.

\subsection{Mean-field limit}
\label{sub: mean-field limit}

Consider discrete controls $\phi_n$ associated with a fluid control $\phi$. As observed earlier, the joint distribution of the state and state-action frequency processes associated with any policy that realizes $\phi_n$ is completely determined by $\phi_n$ and the initial state frequencies. Further, $\bx_n$ is a discrete-time Markov chain and $\by_n(t) = \phi_n\left(\bx_n(t)\right)$ for all $t \geq 0$. Below we provide a convenient construction of these two processes.

For this purpose we consider a collection $\set{\condr{B_t^l}*{j | i, a}}{l \geq 1, t \geq 1,  i, j \in S, a \in A}$ of Bernoulli random variables with the following two properties. First,
\begin{equation*}
	\condr{B_t^l}*{j | i, a} = 1 \quad \text{and} \quad \condr{B_t^l}*{k | i, a} = 0 \quad \text{for all} \quad k \neq j \quad \text{with probability} \quad \condr{p}*{j | i, a}
\end{equation*}
for all $l \geq 1$, $t \geq 1$, $i \in S$ and $a \in A$. Second, the random vectors $\vct{\condr{B_t^l}*{j | i, a}}{j \in S}$ are independent across $l \geq 1$, $t \geq 1$, $i \in S$ and $a \in A$. The same collection of random variables can be used to construct the state and state-action frequency processes for any given $n$, and thus the notation does not account for the number of processes $n$.

At each time $t$, there are $n\by_n(t, i, a)$ processes in state $i$ with action $a$ selected, which we arrange according to their indexes. We postulate that the process with the $l$th largest index is in state $j$ at $t + 1$ if $\condr{B_{t + 1}^l}*{j | i, a} = 1$, which endows $\bx_n(t + 1)$ with the desired distribution. More specifically, we construct $\bx_n$ and $\by_n$ recursively as functions of $\bx_n(0)$ and the above Bernoulli random variables, such that for each $t \geq 0$ we have:
\begin{equation*}
	\by_n(t) = \phi_n\left(\bx_n(t)\right) \quad \text{and} \quad \bx_n(t + 1, j) = \frac{1}{n}\sum_{i \in S} \sum_{a \in A} \sum_{l = 1}^{n\by_n(t, i, a)} \condr{B_{t + 1}^l}*{j | i, a} \quad \text{for all} \quad j \in S.
\end{equation*}

We further define
\begin{equation*}
	\bz_n(t + 1, j) \defeq \frac{1}{n}\sum_{i \in S} \sum_{a \in A} \sum_{l = 1}^{n\by_n(t, i, a)} \left[\condr{B_{t + 1}^l}*{j | i, a} - \condr{p}*{j | i, a}\right]
\end{equation*}
for all $t \geq 0$ and $j \in S$. Then we can write
\begin{equation}
	\label{eq: equation for x_n, y_n and z_n}
	\bx_n(t + 1) = \bz_n(t + 1) + \sum_{a \in A} \by_n(t, a)P(a) \quad \text{for all} \quad t \geq 0.
\end{equation}

Clearly, $\bz_n(t)$ has mean zero for all $t \geq 1$; it is possible to check that $\bz_n$ is in fact a martingale difference with respect to the natural filtration, but this will not be used explicitly. The next lemma bounds the variance of $\bz_n(t)$ and is proved in Appendix \ref{app: proofs of various results}.

\begin{lemma}
	\label{lem: variance of z_n}
	Consider the constants
	\begin{equation*}
		q_{\min}(j) \defeq \min_{i \in S, a \in A} \left[1 - \condr{p}*{j | i, a}\right] \leq \max_{i \in S, a \in A} \left[1 - \condr{p}*{j | i, a}\right] \eqdef q_{\max}(j) \quad \text{for all} \quad j \in S.
	\end{equation*}
	For each $t \geq 1$ and $j \in S$, we have
	\begin{equation*}
		\frac{q_{\min}(j)}{n}\conde*{\bx_n(t, j)} \leq \conde*{\bz_n^2(t, j)} \leq \frac{q_{\max}(j)}{n}\conde*{\bx_n(t, j)}.
	\end{equation*}
\end{lemma}

The lemma will be used in the proof of the following proposition, which establishes that the fluid trajectories defined by the fluid control $\phi$ are the mean-field limit of the state and state-action frequency processes $\bx_n$ and $\by_n$ determined by $\phi_n$. More precisely, the latter processes converge weakly to fluid trajectories provided that the initial state frequency vectors $\bx_n(0)$ have a limit in distribution as $n \to \infty$.

\begin{proposition}
	\label{prop: mean-field limit}
	Assume that $\phi$ is continuous and there exists a random variable $x^0$ with values in $X$ such that $\bx_n(0) \Rightarrow x^0$ as $n \to \infty$. Furthermore, suppose that $x^0$ is defined on some set $\Omega$ and for each $\omega \in \Omega$ let $\set{\bx(\omega, t), \by(\omega, t)}{t \geq 0}$ be the fluid trajectory with $\bx(\omega, 0) = x^0(\omega)$, which defines two stochastic process. Then we have
	\begin{equation*}
		\bx_n(t) \Rightarrow \bx(t) \quad \text{and} \quad \by_n(t) \Rightarrow \by(t) \quad \text{as} \quad n \to \infty \quad \text{for all} \quad t \geq 0.
	\end{equation*}
\end{proposition}

\begin{proof}
	We proceed by induction, assuming that $\bx_n(t) \Rightarrow \bx(t)$ as $n \to \infty$ and then showing that $\by_n(t) \Rightarrow \by(t)$ and $\bx_n(t + 1) \Rightarrow \bx(t + 1)$. Clearly, the assumption holds for $t = 0$.
	
	Assume that $\bx_n(t) \Rightarrow \bx(t)$ as $n \to \infty$ and let $\map{\zeta}{\R^{S \times A} \times X}{\R^{S \times A}}$ be defined by $\zeta(w, x) \defeq w + \phi(x)$. If we let $\bw_n(t) \defeq \phi_n\left(\bx_n(t)\right) - \phi\left(\bx_n(t)\right)$, then $\by_n(t) = \zeta\left(\bw_n(t), \bx_n(t)\right)$. Moreover, $\bw_n(t) \Rightarrow 0$ as $n \to \infty$ by Definition \ref{def: fluid heuristic}, because
	\begin{equation*}
		\norm{\bw_n(t)} \leq \max_{x \in X_n} \norm{\phi_n(x) - \phi(x)}.
	\end{equation*}
	By \cite[Theorem 3.1]{billingsley1999convergence}, we have $\left(\bw_n(t), \bx_n(t)\right) \Rightarrow \left(0, \bx(t)\right)$ as $n \to \infty$; see \cite[Lemma 9]{goldsztajn2023sparse} for further details. The continuous mapping theorem yields $\by_n(t) \Rightarrow \zeta\left(0, \bx(t)\right) = \by(t)$.
	
	Lemma \ref{lem: variance of z_n} implies that $\bz_n(t + 1) \Rightarrow 0$ as $n \to \infty$ since
	\begin{equation*}
		\condp*{\norm{\bz_n(t + 1)}_2^2 \geq \varepsilon} \leq \frac{\conde*{\norm{\bz_n(t + 1)}_2^2}}{\varepsilon} \leq \frac{\conde*{\norm{\bx_n(t)}_1}}{n \varepsilon} \leq \frac{1}{n \varepsilon} \quad \text{for all} \quad \varepsilon > 0.
	\end{equation*}
	Therefore, $\left(\by_n(t), \bz_n(t + 1)\right) \Rightarrow \left(\by(t), 0\right)$ as $n \to \infty$ by \cite[Theorem 3.1]{billingsley1999convergence}. Consider now the continuous function $\map{\eta}{\R^{S \times A} \times \R^S}{\R^S}$ such that
	\begin{equation*}
		\eta(y, z) \defeq z + \sum_{a \in A} y(a)P(a).
	\end{equation*}
	Then $\bx_n(t + 1) = \eta\left(\by_n(t), \bz_n(t + 1)\right)$ by \eqref{eq: equation for x_n, y_n and z_n}, and it follows from the continuous mapping theorem that $\bx_n(t + 1) \Rightarrow \eta\left(\by(t), 0\right) = \bx(t + 1)$ as $n \to \infty$.
\end{proof}

\subsection{Asymptotic optimality result}
\label{sub: asymptotic optimality}

The next lemma assumes that the fluid control $\phi$ is continuous and satisfies \eqref{eq: sufficient asymptotic optimality condition for fluid control}. Under these assumptions, the lemma establishes that the gain achieved by the discrete control $\phi_n$ approaches the optimal value of the fluid relaxation $g_r$ as $n \to \infty$ whenever $\bx_n(0)$ has a stationary distribution for each $n$; at least one stationary distribution for the Markov chain $\bx_n$ always exists since its state space is finite, but this distribution may not be unique. The proof of the lemma relies on the mean-field limit obtained in Proposition \ref{prop: mean-field limit}.

\begin{lemma}
	\label{lem: stationary rewards}
	Assume that $\phi$ is continuous and that every fluid trajectory satisfies
	\begin{equation*}
		\lim_{t \to \infty} \sum_{a \in A} \by(t, a) r(a) = g_r,
	\end{equation*}
	regardless of the initial condition. Also, assume that $\bx_n(0)$ is distributed as a stationary distribution of $\bx_n$, and thus $\bx_n$ is stationary. Then we have
	\begin{equation}
		\label{eq: limit of stationary reward}
		\lim_{n \to \infty} \sum_{a \in A} \conde*{\by_n(t, a)}r(a) = g_r \quad \text{for all} \quad t \geq 0.
	\end{equation}
\end{lemma}

\begin{proof}
	Consider the following limits in distribution:
	\begin{equation}
		\label{aux: weak limit of stationary gain}
		\sum_{a \in A} \by_n(t, a)r(a) \Rightarrow g_r \quad \text{as} \quad n \to \infty \quad \text{for all} \quad t \geq 0.
	\end{equation}
	For each $t \geq 0$, the sequence of random variables on the left-hand side, with $t$ fixed and $n$ varying, is uniformly integrable since it is uniformly bounded:
	\begin{equation}
		\label{aux: gain upper bound}
		\left|\sum_{a \in A} \by_n(t, a)r(a)\right| \leq |S| |A| \max_{i \in S, a \in A} |r(i, a)|.
	\end{equation}
	We conclude that \eqref{aux: weak limit of stationary gain} implies \eqref{eq: limit of stationary reward}, so it suffices to prove \eqref{aux: weak limit of stationary gain}. Because $\bx_n$ is stationary, the following random variables are identically distributed:
	\begin{equation*}
		\sum_{a \in A} \by_n(t, a)r(a) = \sum_{a \in A} \phi_n\left(\bx_n(t), a\right)r(a) \sim \sum_{a \in A} \phi_n(\bx_n(0), a)r(a) = \sum_{a \in A} \by_n(0, a)r(a).
	\end{equation*}
	Hence, it is enough to prove that \eqref{aux: weak limit of stationary gain} holds for $t = 0$.
	
	It follows from \eqref{aux: gain upper bound} that
	\begin{equation*}
		\set{\sum_{a \in A} \by_n(0, a)r(a)}{n \geq 1}
	\end{equation*}
	is tight, and thus every subsequence has a further subsequence that converges weakly. We will establish \eqref{aux: weak limit of stationary gain} for $t = 0$ by showing that the weak limit of every convergent subsequence is $g_r$, i.e., the random variable equal to $g_r$ with probability one.
	
	Suppose that $\calL \subset \N$ and $g_\calL$ are such that
	\begin{equation}
		\label{aux: limit of subsequence}
		\sum_{a \in A} \by_l(0, a)r(a) \Rightarrow g_\calL \quad \text{as} \quad l \to \infty \quad \text{while} \quad l \in \calL.
	\end{equation}
	More specifically, $g_\calL$ is the limit associated with the convergent subsequence indexed in $\calL$, and the subscript indicates that this limit could in principle depend on the specific set of indexes $\calL$. The random variables $\set{\bx_l(0)}{l \in \calL}$ are tight because they are all supported on the compact set $X$. Therefore, we may assume without any loss of generality that there exists a random variable $x_\calL$ such that $\bx_l(0) \Rightarrow x_\calL$ as $l \to \infty$.
	
	Define $\map{\xi}{\R \times \R^S}{\R}$ by
	\begin{equation*}
		\xi(w, x) \defeq w + \sum_{a \in A} \phi(x, a)r(a).
	\end{equation*}
	Note that $\xi$ is a continuous function and satisfies that
	\begin{equation*}
		\sum_{a \in A} \by_l(0, a)r(a) = \xi\left(\sum_{a \in A} \phi_l\left(\bx_l(0), a\right)r(a) - \sum_{a \in A} \phi\left(\bx_l(0), a\right)r(a), \bx_l(0)\right).
	\end{equation*}
	The first argument of $\xi$ converges to zero in probability as $l \to \infty$ by Definition \ref{def: fluid heuristic}, because
	\begin{equation*}
		\left|\sum_{a \in A} \phi_l\left(\bx_l(0), a\right)r(a) - \sum_{a \in A} \phi\left(\bx_l(0), a\right)r(a)\right| \leq \max_{x \in X_l} \norm{\phi_l(x) - \phi(x)}\sum_{a \in A} \norm{r(a)}.
	\end{equation*}
	It follows from \cite[Theorem 3.1]{billingsley1999convergence} and the continuous mapping theorem that
	\begin{equation*}
		\sum_{a \in A} \by_l(0, a)r(a) \Rightarrow \xi\left(0, x_\calL\right) = \sum_{a \in A} \phi(x_\calL, a)r(a) \quad \text{as} \quad l \to \infty \quad \text{while} \quad l \in \calL.
	\end{equation*}
	We conclude from \eqref{aux: limit of subsequence} that the following equality in distribution holds:
	\begin{equation*}
		g_\calL \sim \sum_{a \in A} \phi\left(x_\calL, a\right)r(a).
	\end{equation*}
	
	Suppose that $x_\calL$ is defined on a set $\Omega$, and for each $\omega \in \Omega$ let $\set{\bx(\omega, t), \by(\omega, t)}{t \geq 0}$ be the fluid trajectory starting at $\bx(\omega, 0) = x_\calL(\omega)$. It follows from Proposition \ref{prop: mean-field limit} that $\bx_l(t) \Rightarrow \bx(t)$ as $l \to \infty$ for all $t \geq 0$. Because $\bx_l$ is a stationary process, $\bx_l(t) \sim \bx_l(0)$ for each $t \geq 0$, and this implies that $\bx(t) \sim x_\calL$. We conclude that
	\begin{equation*}
		\sum_{a \in A} \by(t, a)r(a) = \sum_{a \in A} \phi\left(\bx(t), a\right)r(a) \sim \sum_{a \in A} \phi\left(x_\calL, a\right)r(a) \sim g_\calL \quad \text{for all} \quad t \geq 0.
	\end{equation*}
	
	By assumption,
	\begin{equation*}
		\lim_{t \to \infty} \sum_{a \in A} \by(t, a)r(a) = g_r \quad \text{with probability one}.
	\end{equation*}
	In particular, the limit also holds in distribution, which implies that $g_\calL \sim g_r$, i.e., $g_\calL = g_r$ with probability one. This completes the proof of the lemma.
\end{proof}

We can now prove the asymptotic optimality result.

\begin{theorem}
	\label{the: asymptotic optimality}
	Suppose that $\phi$ is continuous and that every fluid trajectory satisfies
	\begin{equation*}
		\lim_{t \to \infty} \sum_{a \in A} \by(t, a) r(a) = g_r,
	\end{equation*}
	regardless of the initial condition. For arbitrary and possibly random vectors $\bx_n(0)$ of initial state frequencies, the gain achieved by the discrete controls $\phi_n$ approaches the optimal value $g_r$ of the fluid relaxation as $n \to \infty$. Specifically,
	\begin{equation*}
		\lim_{n \to \infty} \lim_{T \to \infty} \frac{1}{T}\sum_{t = 0}^{T - 1} \sum_{a \in A} \conde*{\by_n(t, a)}r(a) = g_r.
	\end{equation*}
\end{theorem}

\begin{proof}
	By \cite[Proposition 8.1.1]{puterman2014markov}, there exist random vectors $x_n^s$ such that the distribution of $x_n^s$ is stationary for the Markov chain $\bx_n$ and satisfies that
	\begin{equation*}
		\lim_{T \to \infty} \frac{1}{T}\sum_{t = 0}^{T - 1} \sum_{a \in A} \conde*{\by_n(t, a)}r(a) = \sum_{a \in A} \conde*{\phi_n\left(x_n^s, a\right)}r(a).
	\end{equation*}
	More precisely, the distribution of $x_n^s$ is a convex combination of the rows of the limiting matrix associated with the Markov chain $\bx_n$. If $\bx_n(0)$ is deterministic, then the distribution of $x_n^s$ is given by the row of the limiting matrix associated with state $\bx_n(0)$. In general, the distribution of $x_n^s$ is the convex combination of the rows such that the weight of the row associated with state $x$ is the probability that $\bx_n(0) = x$. The law of $x_n^s$ is stationary because the rows of the limiting matrix define stationary distributions.
	
	Consider now the stationary Markov chains $\bx_n$ with initial distributions $x_n^s$. Then
	\begin{equation*}
		\lim_{n \to \infty} \sum_{a \in A} \conde*{\phi_n\left(x_n^s, a\right)}r(a) = \lim_{n \to \infty} \sum_{a \in A} \conde*{\by_n(0, a)}r(a) = g_r
	\end{equation*}
	by Lemma \ref{lem: stationary rewards}, which completes the proof.
\end{proof}

The following corollary is a straightforward consequence of Theorem \ref{the: asymptotic optimality}.

\begin{corollary}
	\label{cor: asymptotic optimality}
	Let $y^*$ be an optimal solution of the fluid relaxation \eqref{pr: fluid relaxation}, and suppose that $\phi$ is continuous and such that all fluid trajectories satisfy that
	\begin{equation}
		\label{eq: global attractivity}
		\lim_{t \to \infty} \by(t) = y^*,
	\end{equation}
	regardless of the initial condition. For arbitrary and possibly random initial state frequency vectors $\bx_n(0)$, the limit of the gain achived by $\phi_n$ is given by:
	\begin{equation*}
		\lim_{n \to \infty} \lim_{T \to \infty} \frac{1}{T}\sum_{t = 0}^{T - 1} \sum_{a \in A} \conde*{\by_n(t, a)}r(a) = g_r.
	\end{equation*}
\end{corollary}

\begin{proof}
	Since $y^*$ is an optimal solution of the fluid relaxation and $\by(t) \to y^*$ as $t \to \infty$ for all fluid trajectories, the assumptions of Theorem \ref{the: asymptotic optimality} hold.
\end{proof}

In the context of multi-armed restless bandits, global attractivity properties that are analogous to \eqref{eq: global attractivity} have been used to establish the asymptotic optimality of the Whittle index and the LP-priority policies, in \cite{weber1990index,gast2023exponential} and \cite{verloop2016asymptotically,gast2023linear}, respectively; further conditions are also required, such as the single-arm and multi-arm problems being unichain. These global attractivity properties must be checked for each problem instance and there are remarkably simple instances where the property fails. One example, where the single-arm problem has three states and is ergodic, is given in \cite{gast2020exponential} and is considered in Section \ref{sec: examples}.

In the following section we will construct continuous fluid controls that satisfy \eqref{eq: global attractivity} when a technical condition holds. Unlike the global attractivity properties required by the Whittle index and LP-priority policies, our condition holds automatically for a broad class of problems. In particular, the condition always holds for multi-armed restless bandit problems where the unconstrained single-arm problem admits a policy that is unichain, aperiodic and such that its unique irreducible class contains the set defined below.

\begin{definition}
	\label{def: good states}
	Given an optimal solution $y^*$ of the fluid relaxation \eqref{pr: fluid relaxation}, let
	\begin{equation}
		\label{eq: definition of x^* and S_+^*}
		x^* = \sum_{a \in A} y^*(a) \quad \text{and} \quad S_+^* \defeq \set{i \in S}{x^*(i) > 0}.
	\end{equation}
\end{definition}

We emphasize that the above-mentioned policy for the single-arm problem need not comply with the constraints in \eqref{eq: constraints} nor maximize the accumulated reward in any way. The only requirements are that the policy is unichain, aperiodic and such that all transient states are outside of $S_+^*$. Moreover, we establish that such a policy exists if and only if the policy that selects actions uniformly at random has the latter properties. Therefore, our asymptotic optimality condition is straightforward to check.

\section{Fluid controls}
\label{sec: fluid controls}

Consider an optimal solution $y^*$ of the fluid relaxation and define
\begin{equation*}
	L(y) \defeq \sum_{a \in A} y(a)P(a) \quad \text{and} \quad x^* \defeq \sum_{a \in A} y^*(a) = L\left(y^*\right) \in X.
\end{equation*}
Note that $\map{L}{\R^{S \times A}}{\R^S}$ is linear and that the latter equality holds since $y^*$ is a solution of the fluid relaxation. In addition, define $\map{\beta}{\R_+^S}{[0, \infty)}$ by
\begin{equation*}
	\beta(x) \defeq \max \set{\lambda \geq 0}{\lambda x^* \preceq x}.
\end{equation*}

We prove the next lemma in Appendix \ref{app: proofs of various results}.

\begin{lemma}
	\label{lem: function beta}
	The following properties hold.
	\begin{enumerate}
		\item[(a)] If $x \in X$ and $x \neq x^*$, then $\left[1 - \beta(x)\right]^{-1}\left[x - \beta(x)x^*\right] \in X$.
		
		\item[(b)] The function $\beta$ is continuous.
	\end{enumerate}
\end{lemma}

Given an auxiliary  fluid control $\psi$, we can define another fluid control $\phi$ by letting
\begin{equation}
	\label{eq: fluid control structure}
	\phi(x) \defeq \beta(x) y^* + \left[1 - \beta(x)\right] \psi\left(\left[1 - \beta(x)\right]^{-1}\left[x - \beta(x)x^*\right]\right) \quad \text{for all} \quad x \in X;
\end{equation}
here $\phi\left(x^*\right) \defeq y^*$, i.e., the second term on the right-hand side is zero if $\beta(x) = 1$. It is straightforward to check that $\phi$ is indeed a fluid control when $\psi$ is.

\begin{remark}
	\label{rem: comparison with yan}
	The \emph{align and steer} policy considered in \cite{yan2024optimal} has the same structure as the fluid controls $\phi$ defined by \eqref{eq: fluid control structure}; the \emph{maximum alignment coefficient} and $\pi_{\text{steer}}$ policy in \cite{yan2024optimal} correspond to our $\beta$ and $\psi$, respectively. However, the results in \cite{yan2024optimal} concern multi-armed restless bandits and assume that $\pi_{\text{steer}}$ has a specific form which is linear, whereas we study problems with multiple actions and constraints and let $\psi$ be general and possibly nonlinear. Under these general assumptions, we provide below a necessary and sufficient condition on $\psi$ for the global attractivity property \eqref{eq: global attractivity} to hold. 
\end{remark}

Note that every fluid trajectory satisfies that
\begin{equation*}
	\by(t) = \phi\left(\bx(t)\right) \quad \text{and} \quad \bx(t + 1) = L\left(\by(t)\right) = L\left(\phi\left(\bx(t)\right)\right) \quad \text{for all} \quad t \geq 0.
\end{equation*}
Furthermore, \eqref{eq: global attractivity} holds if and only if $\bx(t) \to x^*$ as $t \to \infty$, and the latter property holds if and only if $\beta\left(\bx(t)\right) \to 1$. In the next sections we provide conditions on $\psi$ for these properties to hold. First we consider problems with general constraints and then we provide more direct conditions for specific sets of constraints.

\subsection{General constraints}
\label{sub: general constraints}

We adopt the following notation for the composition of functions:
\begin{equation*}
	\eta \circ \zeta(x) \defeq \eta\left(\zeta(x)\right), \quad \zeta^0(x) \defeq x \quad \text{and} \quad \zeta^k(x) \defeq \zeta \circ \zeta^{k - 1}(x) \quad \text{for all} \quad k \geq 1,
\end{equation*}
whenever the expressions on the right-hand sides are well-defined.

\begin{condition}
	\label{cond: general}
	The fluid control $\psi$ is continuous. Moreover, consider the set
	\begin{equation*}
		Z \defeq \set{x \in X}{\beta(x) = 0} = \set{x \in X}{x(i) = 0\ \text{for some}\ i \in S_+^*}.
	\end{equation*}
	For each $z \in Z$, there exists $t \geq 1$ such that $\left(L \circ \psi\right)^t(z) \notin Z$.
\end{condition}

Observe that the assumptions of Corollary \ref{cor: asymptotic optimality} cannot be satisfied if the fluid control does not satisfy Condition \ref{cond: general}. Indeed, a fluid control that does not satisfy Condition \ref{cond: general} is either discontinuous or such that some fluid trajectory is contained in $Z$ and thus does not satisfy \eqref{eq: global attractivity}. The following theorem proves that the existence of a fluid control that satisfies Condition \ref{cond: general} is not only necessary for Corollary \ref{cor: asymptotic optimality} but also sufficient.

\begin{theorem}
	\label{the: optimal fluid control general case}
	The existence of a fluid control that satisfies Condition \ref{cond: general} is necessary and sufficient for Corollary \ref{cor: asymptotic optimality}. Moreover, if $\psi$ satisfies Condition \ref{cond: general} and $\phi$ is as in \eqref{eq: fluid control structure}, then $\phi$ is a continuous fluid control such that \eqref{eq: global attractivity} holds, and thus Corollary \ref{cor: asymptotic optimality} holds.
\end{theorem}

\begin{proof}
	As noted earlier, it is straightforward to check that $\phi$ is a fluid control. Further, Lemma \ref{lem: function beta} and Condition \ref{cond: general} imply that $\phi$ is continuous at $x$ for all $x \neq x^*$. In order to prove that $\phi$ is continuous at $x^*$, suppose that $x_k \to x^*$ as $k \to \infty$ with $x_k \in X \setminus \left\{x^*\right\}$. Then
	\begin{equation*}
		\lim_{k \to \infty} \beta\left(x_k\right) = 1 \quad \text{and} \quad \norm{\psi\left(\left[1 - \beta(x_k)\right]^{-1}\left[x_k - \beta(x_k)x^*\right]\right)}_1 = 1 \quad \text{for all} \quad k \geq 1.
	\end{equation*}
	It follows that $\phi(x_k) \to y^* = \phi\left(x^*\right)$ as $k \to \infty$, and thus $\phi$ is continuous.
	
	It only remains to prove that every fluid trajectory satisfies that $\by(t) \to y^*$ as $t \to \infty$, and for this purpose it suffices to show that $\beta\left(\bx(t)\right) \to 1$ as $t \to \infty$. We will do this by assuming that the latter property does not hold and arriving to a contradiction.
	
	Suppose then that $\beta\left(\bx(t)\right) \nrightarrow 1$ as $t \to \infty$ for some fluid trajectory, and note that the function $t \mapsto \beta\left(\bx(t)\right)$ is nondecreasing since
	\begin{align*}
		\bx(t + 1) &= L\circ\phi\left(\bx(t)\right) \\
		&= \beta\left(\bx(t)\right)x^* + L\left(\left[1 - \beta\left(\bx(t)\right)\right] \psi\left(\left[1 - \beta\left(\bx(t)\right)\right]^{-1}\left[\bx(t) - \beta\left(\bx(t)\right)x^*\right]\right)\right),
	\end{align*}
	where the last term is in $X \subset \R_+^S$. Therefore, there exists
	\begin{equation}
		\label{aux: definition of gamma}
		\gamma \defeq \lim_{t \to \infty} \beta\left(\bx(t)\right) < 1.
	\end{equation}
	
	Because $X$ is a compact set, there exist $x_\gamma \in X$ and some sequence $\set{t_k \in \N}{k \geq 1}$ such that we have $\bx\left(t_k\right) \to x_\gamma$ as $k \to \infty$. Furthermore,
	\begin{equation*}
		\beta\left(x_\gamma\right) = \lim_{k \to \infty} \beta\left(\bx\left(t_k\right)\right) = \gamma.
	\end{equation*}
	Let $\set{\bx_\gamma(t), \by_\gamma(t)}{t \in \N}$ be the fluid trajectory such that $\bx_\gamma(0) = x_\gamma$. Then
	\begin{equation}
		\label{aux: stucked at gamma}
		\beta\left(\bx_\gamma(t)\right) = \gamma \quad \text{for all} \quad t \in \N.
	\end{equation}
	Indeed, assume that $\beta\left(\bx_\gamma(t)\right) > \gamma$ for some $t \geq 0$, and note that
	\begin{equation*}
		\lim_{k \to \infty} \bx\left(t + t_k\right) = \lim_{k \to \infty} \left(L \circ \phi\right)^t \left(\bx\left(t_k\right)\right) = \left(L \circ \phi\right)^t \left(x_\gamma\right) = \bx_\gamma(t)
	\end{equation*}
	because $\phi$ is continuous. If \eqref{aux: stucked at gamma} failed, then \eqref{aux: definition of gamma} would fail as well since
	\begin{equation*}
		\lim_{k \to \infty} \beta\left(\bx\left(t + t_k\right)\right) = \beta\left(\bx_\gamma(t)\right) > \gamma.
	\end{equation*}
	
	Define $\bz(t) \defeq \left(1 - \gamma\right)^{-1}\left[\bx_\gamma(t) - \gamma x^*\right] \in Z$ for all $t \in \N$. Then
	\begin{align*}
		\bz(t + 1) &= \left(1 - \gamma\right)^{-1}\left[\bx_\gamma(t + 1) - \gamma x^*\right] \\
		&= \left(1 - \gamma\right)^{-1}\left[L\circ\phi\left(\bx_\gamma(t)\right) - \gamma x^*\right] \\
		&= \left(1 - \gamma\right)^{-1}\left[\gamma x^* + L\left(\left(1 - \gamma\right)\psi\left(\left(1 - \gamma\right)^{-1}\left[\bx_\gamma(t) - \gamma x^*\right]\right)\right) - \gamma x^*\right] \\
		&= L \circ \psi\left(\bz(t)\right) = \left(L\circ\psi\right)^{t + 1}\left(\bz(0)\right) \quad \text{for all} \quad t \in \N,
	\end{align*}
	where we used the linearity of $L$. Condition \ref{cond: general} implies that $\bz(t) \notin Z$ for some $t \geq 1$, but this is only possible if $\beta\left(\bx_\gamma(t)\right) > \gamma$ and thus contradicts \eqref{aux: stucked at gamma}.
\end{proof}

The following definition can be helpful for checking Condition \ref{cond: general}.

\begin{definition}
	\label{def: fluid control based on single-process policy}
	Let $\pi \defeq \set{\condr{\pi}*{a | i} \in [0, 1]}{i \in S, a \in A}$ be a set of constants such that
	\begin{equation*}
		\sum_{a \in A} \condr{\pi}*{a | i} = 1 \quad \text{for all} \quad i \in S.
	\end{equation*}
	Consider diagonal matrices $\set{D_\pi(a)}{a \in A}$ such that $D_\pi(i, i, a) \defeq \condr{\pi}*{a | i}$ for all $i \in S$ and $a \in A$. We may define a function $\map{\psi_1}{X}{Y}$ by letting
	\begin{equation*}
		\psi_1(x, a) \defeq xD_\pi(a) \quad \text{for all} \quad x \in X \quad \text{and} \quad a \in A,
	\end{equation*}
	where $\psi_1(x, a)$ is a short-hand notation for $y(a)$ when $\psi_1(x) = y$. The fluid control $\psi$ is purely based on $\pi$ if $\psi = \psi_1$, and we say that $\psi$ is partially based on $\pi$ if there exist a constant $\gamma \in (0, 1)$ and a function $\map{\psi_2}{X}{Y}$ such that
	\begin{equation}
		\label{eq: partially based}
		\psi(x) = \gamma \psi_1(x) + (1 - \gamma)\psi_2(x) \quad \text{for all} \quad x \in X.
	\end{equation}
\end{definition}

The set of constants $\pi$ corresponds to a stationary policy for a single-process problem without constraints. In particular, $\condr{\pi}*{a | i}$ can be interpreted as the probability that action $a$ is selected whenever the process is in state $i$. If $\psi$ is purely based on $\pi$, then
\begin{equation*}
	x\sum_{a \in A} D_\pi(a)C(a) = d \quad \text{and} \quad x\sum_{a \in A} D_\pi(a)E(a) \preceq f 
\end{equation*}
must hold for all $x \in X$ for $\psi$ to be a fluid control. However, these conditions need not hold when $\psi$ is only partially based on $\pi$ as in \eqref{eq: partially based} because the function $\psi_2$ can help to satisfy the conditions in Definition \ref{def: fluid control}, particularly the constraints of the problem.

\begin{condition}
	\label{cond: fluid control partially based on pi}
	The fluid control $\psi$ is continuous and partially based on a set of constants $\pi$ as in Definition \ref{def: fluid control based on single-process policy}. Consider the matrix $P_\pi$ such that
	\begin{equation*}
		P_\pi(i, j) \defeq \sum_{a \in A} \condr{\pi}*{a | i} \condr{p}*{j | i, a} \quad \text{for all} \quad i, j \in S.
	\end{equation*}
	This matrix is unichain, aperiodic and such that $S_+^*$ is inside the unique irreducible class. 
\end{condition}

\begin{remark}
	\label{rem: existence of single-process policy}
	Suppose that a policy $\pi$ for the unconstrained single-process problem has the above properties and consider the policy $\nu$ that always selects each action with the same probability, regardless of the current state. It is easy to check that this policy also has the properties listed in Condition \ref{cond: fluid control partially based on pi} since the arrows in the transitions diagram for $\pi$ are a subset of the arrows in the transitions diagram for $\nu$. Therefore, there exists a policy $\pi$ that is unichain, aperiodic and such that the unique irreducible class contains $S_+^*$ if and only if the specific policy $\nu$ defined earlier has these properties.
\end{remark}

The following corollary follows from Theorem \ref{the: optimal fluid control general case}.

\begin{corollary}
	\label{cor: optimal fluid control partially based on some pi}
	Let $\psi$ be a fluid control satisfying Condition \ref{cond: fluid control partially based on pi} and $\phi$ be as in \eqref{eq: fluid control structure}. Then $\phi$ satisfies the assumptions of Corollary \ref{cor: asymptotic optimality}.
\end{corollary}

\begin{proof}
	It is enough to establish that Condition \ref{cond: fluid control partially based on pi} implies Condition \ref{cond: general}. For this purpose we will prove that for all $z \in X$ and $k \geq 0$, we have
	\begin{equation}
		\label{aux: composition of l and psi}
		\left(L \circ \psi\right)^k(z) = \gamma^k z P_\pi^k + (1 - \gamma) w_k(z) \quad \text{with} \quad w_k(z) \in \R_+^S.
	\end{equation}
	Condition \ref{cond: fluid control partially based on pi} implies that $P_\pi$ has a unique stationary distribution $x_\pi \in X$ that assigns positive probability to all $i \in S_+^*$. Furthermore, $z P_\pi^k \to x_\pi$ as $k \to \infty$ for each initial distribution $z \in X$. Therefore, there always exists $k \geq 1$ such that the right-hand side of \eqref{aux: composition of l and psi} assigns positive probability to all $i \in S_+^*$, and thus $\left(L \circ \psi\right)^k(z) \notin Z$.
	
	It only remains to prove that \eqref{aux: composition of l and psi} holds for all $k \geq 0$. For this purpose fix $z \in X$ and note that \eqref{aux: composition of l and psi} holds for $k = 0$. We now proceed by induction, assuming that \eqref{aux: composition of l and psi} holds for some $k$ and showing that it then holds for $k + 1$ as well. First note that
	\begin{equation*}
		\left(L \circ \psi\right)^{k + 1}(z) = \gamma L\left(\psi_1\left(\left(L \circ \psi\right)^k(z)\right)\right) + (1 - \gamma) L\left(\psi_2\left(\left(L \circ \psi\right)^k(z)\right)\right). 
	\end{equation*}
	
	We may compute the first term on the right-hand side by noting that
	\begin{equation*}
		L\circ\psi_1(x) = \sum_{a \in A} x D_\pi(a)P(a) = x P_\pi \quad \text{for all} \quad x \in X.
	\end{equation*}
	Since \eqref{aux: composition of l and psi} holds for $k$, the previous equation implies that
	\begin{align*}
		\gamma L\left(\psi_1\left(\left(L \circ \psi\right)^k(z)\right)\right) &= \gamma L\left(\psi_1\left(\gamma^k z P_\pi^k + (1 - \gamma) w_k(z)\right)\right) \\
		&= \gamma^{k + 1} z P_\pi^{k + 1} + \gamma (1 - \gamma)w_k(z) P_\pi.
	\end{align*}
	It follows that \eqref{aux: composition of l and psi} holds for $k + 1$ with
	\begin{equation*}
		w_{k + 1}(z) \defeq \gamma (1 - \gamma)w_k(z) P_\pi + (1 - \gamma) L\left(\psi_2\left(\left(L \circ \psi\right)^k(z)\right)\right).
	\end{equation*}
	The first term on the right-hand side is in $\R_+^S$ since $w_k(z) \in \R_+^S$ and $P_\pi$ has nonnegative entries. Also, the second term is in $\R_+^S$ as well since $L \circ \psi$ and $L \circ \psi_2$ map $X$ into $X$.
\end{proof}

Suppose that the fluid control $\phi$ is defined as in \eqref{eq: fluid control structure} in terms of some fluid control $\psi$ that satisfies Condition \ref{cond: fluid control partially based on pi}. The functions $\psi_1$ and $\psi_2$ defining $\psi$ play different roles in ensuring that the assumptions of Corollary \ref{cor: asymptotic optimality} hold. Informally speaking, $\psi_1$ is responsible for preventing $\beta\left(\bx(t)\right)$ from getting stuck at a value strictly smaller than one when $\bx$ is a fluid trajectory, which is necessary for \eqref{eq: global attractivity}. On the other hand, $\psi_2$ ensures that $\psi$ complies with the constraints of the problem and thus is indeed a fluid control.

The function $\psi_1$ is completely determined by the constants $\pi$, which correspond to a stationary policy for a single process without constraints. Such a policy can be found by inspecting the transitions diagrams of a single process for the different actions. Moreover, as noted in Remark \ref{rem: existence of single-process policy}, there exists $\pi$ that satisfies Condition \ref{cond: fluid control partially based on pi} if and only if the policy $\nu$ that selects actions uniformly at random satisfies the condition. In contrast, the choice of $\psi_2$ depends on the constraints of the problem. In the following two sections we provide rules for defining $\psi_2$ when the constraints have specific structures. Nonetheless, we remark that asymptotically optimal policies can be obtained for problems with different constraint structures if a expression for $\psi_2$, such that $\psi$ satisfies the constraints, is available.

\subsection{Inequality constraints}
\label{sub: inequality constraints}

In this section we impose the following assumption on the constraints.

\begin{assumption}
	\label{ass: inequality constraints}
	The equality constraints are trivial, i.e., we have $C(a) = 0$ for all $a \in A$ and $d = 0$. In addition, the matrices $E(a)$ are nonnegative, the vector $f$ has strictly positive entries and there exists $0 \in A$ such that $E(0) = 0$.
\end{assumption}

As noted in Section \ref{sec: problem formulation}, constraints with nonnegative coefficients as in Assumption \ref{ass: inequality constraints} can be interpreted as resource allocation constraints. Besides nonnegative coefficients, the assumption imposes two further conditions. The first condition is that there are no equality constraints, which means that it is not compulsory to fully consume any of the resources. The second condition is that there exists an action $0 \in A$ that does not consume any of the resources, and can often be interpreted as doing nothing. In this setup, a fluid control that satisfies the assumptions of Corollary \ref{cor: asymptotic optimality} exists if the unconstrained single-process problem admits a unichain and aperiodic policy such that the unique irreducible class is contained in $S_+^*$. The following proposition proves this; the proof is provided in Appendix \ref{app: proofs of various results}.

\begin{proposition}
	\label{prop: problems with inequality constraints}
	Suppose that Assumption \ref{ass: inequality constraints} holds and there exists a set of constants $\pi$ as in Definition \ref{def: fluid control based on single-process policy} such that the associated matrix $P_\pi$ satisfies Condition \ref{cond: fluid control partially based on pi}. Define
	\begin{equation*}
		\gamma \defeq \min\set{1, \frac{f(k)}{E(i, k, a)}}{E(i, k, a) \neq 0} \in (0, 1]
	\end{equation*}
	and $\map{\psi_2}{X}{Y}$ such that $\psi_2(x, 0) \defeq x$ and $\psi_2(x, a) \defeq 0$ for $a \neq 0$. Then Condition \ref{cond: fluid control partially based on pi} holds for the mapping $\psi$ defined by
	\begin{equation*}
		\psi(x, a) \defeq \gamma x D_\pi(a) + (1 - \gamma)\psi_2(x, a) \quad \text{for all} \quad x \in X \quad \text{and} \quad a \in A.
	\end{equation*}
\end{proposition}

The proposition gives an expression for a fluid control that satisfies Condition \ref{cond: fluid control partially based on pi}. If the constraints \eqref{eq: constraints} do not depend on $n$, then the following remark provides a simple rounding procedure for defining asymptotically optimal discrete controls.

\begin{remark}
	\label{rem: rounding for inequality constraints}
	Suppose that the assumptions of Proposition \ref{prop: problems with inequality constraints} hold and
	\begin{equation*}
		E_n(a) = E(a) \quad \text{for all} \quad a \in A \quad \text{and} \quad f_n = f \quad \text{for all} \quad n \geq 1.
	\end{equation*}
	Let $\phi$ and $\psi$ be as in \eqref{eq: fluid control structure} and Proposition \ref{prop: problems with inequality constraints}, respectively, and define $\map{\phi_n}{X_n}{Y_n}$ by
	\begin{equation*}
		\phi_n(x) (i, a) \defeq \frac{\floor{n\phi(x) (i, a)}}{n} \quad \text{for all} \quad a \neq 0 \quad \text{and} \quad \phi_n(x) (i, 0) = x(i) - \sum_{a \neq 0} \phi_n(x) (i, a)
	\end{equation*}
	for all $x \in X_n$ and $i \in S$. In addition, note that
	\begin{equation*}
		\max_{x \in X_n} \norm{\phi(x) - \phi_n(x)}_\infty \leq \frac{|A|}{n}. 
	\end{equation*}
	Note that $\phi_n(x)$ puts less weight on the actions $a \neq 0$ than $\phi(x)$, and thus $\phi_n(x)$ satisfies \eqref{eq: constraints} for all $x \in X_n$. Hence, the functions $\phi_n$ are discrete controls for $\phi$. It follows from Corollary \ref{cor: asymptotic optimality} and Proposition \ref{prop: problems with inequality constraints} that these discrete controls are asymptotically optimal.
\end{remark} 

\subsection{Multi-armed restless bandits}
\label{sub: restless bandits}

Next we consider multi-armed restless bandit problems with an equality constraint, as stated in the following assumption. Note that multi-armed restless bandit problems with an inequality constraint are already covered by Assumption \ref{ass: inequality constraints}.

\begin{assumption}
	\label{ass: multi-armed restless bandits with equality constraint}
	The action space is $A = \{0, 1\}$ and there is a unique equality constraint:
	\begin{equation*}
		d \in (0, 1), \quad C(i, 0) = 0 \quad \text{and} \quad C(i, 1) = 1 \quad \text{for all} \quad i \in S;
	\end{equation*}
	In addition, the inequality constraints are trivial, i.e., $E(a) = 0$ for all $a \in A$ and $f = 0$.
\end{assumption}

In this case there is a unique resource which must be fully consumed and there are only two actions. As in Section \ref{sub: inequality constraints}, there exists an action $0 \in A$ that does not consume the resource and can often be interpreted as doing nothing. In contrast, the other action consumes one unit of resource per arm, regardless of the state in which the arm is. The following proposition provides a fluid control that satisfies the assumptions of Corollary \ref{cor: asymptotic optimality} if the unconstrained single-arm problem has a unichain and aperiodic policy such that the unique irreducible class is contained in $S_+^*$; the proof is provided in Appendix \ref{app: proofs of various results}.

\begin{proposition}
	\label{prop: multi-armed restless bandits}
	Suppose that Assumption \ref{ass: multi-armed restless bandits with equality constraint} holds and there exists a set of constants $\pi$ as in Definition \ref{def: fluid control based on single-process policy} such that the matrix $P_\pi$ satisfies Condition \ref{cond: fluid control partially based on pi}. Define $\map{\psi_2}{X}{Y}$ by
	\begin{equation*}
		\psi_2(x)(i, 1) \defeq \frac{d - \sum_{j \in S} x(j) d \condr{\pi}*{1 | j}}{(1 - d)\sum_{j \in S} x(j)\left[1 - d \condr{\pi}*{1 | j}\right]}x(i)\left[1 - d \condr{\pi}*{1 | i}\right]
	\end{equation*}
	and $\psi_2(x)(i, 0) \defeq x(i) - \psi_2(x)(i, 1)$ for all $i \in S$. Condition \ref{cond: fluid control partially based on pi} holds with
	\begin{equation*}
		\psi(x, a) \defeq d x D_\pi(a) + (1 - d) \psi_2(x, a) \quad \text{for all} \quad x \in X \quad \text{and} \quad a \in A.
	\end{equation*}
\end{proposition}

The following remark provides a simple rounding procedure for defining asymptotically optimal discrete controls using the fluid control defined above.

\begin{remark}
	\label{rem: rounding for multi-armed restless bandits}
	Suppose that the assumptions of Proposition \ref{prop: multi-armed restless bandits} hold and that
	\begin{equation*}
		d_n = \frac{\floor{d n}}{n} \quad \text{for all} \quad n \geq 1,
	\end{equation*}
	which is typically assumed for multi-armed restless bandit problems. Let $\phi$ and $\psi$ be as in \eqref{eq: fluid control structure} and Proposition \ref{prop: multi-armed restless bandits}, respectively, and define $\map{\phi_n}{X_n}{Y_n}$ as follows. First let
	\begin{equation*}
		\theta_n(x, i) \defeq \floor{d n} - \sum_{i \in S} \floor{n\phi(x)(i, 1)} - \sum_{j = 1}^{i - 1} \ind{n\phi(x)(i, 1) \notin \Z} \quad \text{for all} \quad x \in X_n \quad \text{and} \quad i \in S,
	\end{equation*}
	where we have enumerated the states, i.e., $S = \{1, \dots, |S|\}$. Then define
	\begin{equation*}
		\phi_n(x)(i, 1) \defeq \begin{cases}
			\phi(x)(i, 1) & \text{if} \quad n\phi(x)(i, 1) \in \Z, \\
			\frac{1}{n}\left[\floor{n\phi(x)(i, 1)} + 1\right] & \text{if} \quad n\phi(x)(i, 1) \notin \Z \quad \text{and} \quad \theta_n(x, i) > 0, \\
			\frac{1}{n}\floor{n\phi(x)(i, 1)} & \text{otherwise}.
		\end{cases}
	\end{equation*}
	and $\phi_n(x)(i, 0) \defeq x(i) - \phi_n(x)(i, 1)$ for all $x \in X_n$ and $i \in S$. This can be interpreted algorithmically as follows. In a first stage, we let $\phi_n(x)(i, 1) = \floor{n\phi(x)(i, 1)} / n$  for all $i \in S$. If $n\phi_n(x)(i, 1) \notin \Z$ for some $i$, then we must select action $1$ for more processes in order to exhaust the budget $\floor{dn}$. Therefore, in a second stage, we consider the states $i$ in increasing order and increase $\phi_n(x)(i, 1)$ for some states until the budget is exhausted. We increase $\phi_n(x)(i, 1)$ by $1 / n$ if the budget has not been exhausted yet and $n\phi(x)(i, 1) \notin \Z$. The above expression for $\phi_n(x)(i, 1)$ is obtained by noting that the remaining budget when state $i$ is considered is equal to $\theta_n(x, i)$. It is easy to check that
	\begin{equation*}
		\max_{x \in X_n} \norm{\phi(x) - \phi_n(x)}_\infty \leq \frac{1}{n}
	\end{equation*}
	and that $\phi_n(x)$ satisfies the unique equality constraint for all $x \in X_n$. Thus, the functions $\phi_n$ are discrete controls for $\phi$. We conclude from Corollary \ref{cor: asymptotic optimality} and Proposition \ref{prop: multi-armed restless bandits} that these discrete controls are asymptotically optimal.
\end{remark}

\section{Examples}
\label{sec: examples}

Sections \ref{sub: inequality constraints} and \ref{sub: restless bandits} describe clear procedures for obtaining asymptotically optimal discrete controls in the cases of resource allocation inequality constraints and multi-armed restless bandits, respectively. In both cases the procedure involves the following steps.
\begin{enumerate}
	\item[(a)] Compute an optimal solution $y^*$ of \eqref{pr: fluid relaxation} and define $x^*$ and $S_+^*$ as in \eqref{eq: definition of x^* and S_+^*}.
	
	\item[(b)] Find some policy $\pi$ for the unconstrained single-process problem that is unichain, aperiodic and such that the unique irreducible class contains $S_+^*$.
	
	\item[(c)] Use $\pi$ to define a fluid control $\psi$ as in Proposition \ref{prop: problems with inequality constraints} or \ref{prop: multi-armed restless bandits}, depending on the type of problem, and define a fluid control $\phi$ as in \eqref{eq: fluid control structure}; this solves the fluid problem.
	
	\item[(d)] Compute a discrete control $\phi_n$ by rounding the fluid control $\phi$ as in Remark \ref{rem: rounding for inequality constraints} or \ref{rem: rounding for multi-armed restless bandits}, depending on the type of problem; this solves the approximation problem.
\end{enumerate}

Step (a) can be carried out numerically. Then Remark \ref{rem: existence of single-process policy} can be used to check whether a policy $\pi$ as in step (b) exists, and in that case it is possible to take $\pi = \nu$ with $\nu$ as in Remark \ref{rem: existence of single-process policy}. Nonetheless, another good candidate policy is given by
\begin{equation}
	\label{eq: candidate pi}
	\condr{\mu}{a | i} = \begin{cases}
		\frac{y^*(i, a)}{x^*(i)} & \text{if} \quad x^*(i) > 0, \\
		\frac{1}{|A|} & \text{if} \quad x^*(i) = 0.
	\end{cases}
\end{equation}
We have observed empirically that this policy can perform better than $\nu$ for finite values of $n$ when it satisfies Condition \ref{cond: fluid control partially based on pi}. Also, it has the nice property that $x^* D_\mu(a) = y^*(a)$ for all $a \in A$. The remaining steps (c) and (d) are straightforward to carry out.

In the following two sections we use the above procedure to obtain discrete controls for concrete problem instances. The policy $\mu$ satisfies Condition \ref{cond: fluid control partially based on pi} and is selected in step (b) in all the cases with the exception of the problem instance considered in Section \ref{subsub: counterexample to id policy}. There $\mu$ does not satisfy Condition \ref{cond: fluid control partially based on pi}, but $\nu$ does and is selected in step (b). First we consider a problem with multiple actions and resource allocation inequality constraints as in Section \ref{sub: inequality constraints}, and such that the single-process problem is multichain. Then we consider multi-armed restless bandit problems, focusing on different problem instances where the policies proposed in \cite{whittle1988restless,verloop2016asymptotically,hong2023restless,hong2024unichain} are not asymptotically optimal. The discrete controls $\phi_n$ satisfy the assumptions of Proposition \ref{prop: problems with inequality constraints} or \ref{prop: multi-armed restless bandits} in all the cases and therefore are provably asymptotically optimal, as is confirmed by our numerical experiments.

\subsection{Multiple actions and constraints}
\label{sub: multiple actions and constraints}

Consider a fleet of $n$ electric taxis and suppose that each process describes the battery level of a taxi. Namely, $S = \{0, \dots, 7\}$ are the battery levels and $\bS_n(t, m)$ represents the battery level at time $t$ of taxi $m$. Also, the actions $A = \{0, 1, 2\}$ are: to deploy the taxi at the airport, to deploy the taxi at the city center and to charge the taxi, respectively.

The battery level of a taxi that is being charged increases at a constant rate, whereas that of a deployed taxi decreases by a random amount. More precisely, the distribution of $\bS_n(t + 1, m)$ given that $\bA_n(t, m) = a$ and $\bS_n(t, m) = i$ is known and given by:
\begin{equation*}
	\begin{tabular}{ll}
		$\bS_n(t + 1, m) = \min \left\{i + 2, 7\right\}$ & $\text{if} \quad a = 2$, \\
		$\bS_n(t + 1, m) \sim \max \left\{i - X_a, 0\right\}$ & $\text{if} \quad a \neq 2$.
	\end{tabular}
\end{equation*}
The random variable $X_a$ represents the power consumed by a taxi during period $t$, and tends to be larger for taxis that are deployed at the airport. Specifically, $X_a$ has a Poisson distribution with a mean that depends on $a$ in the following way:
\begin{equation*}
	\condp*{X_a = k} \defeq \frac{\lambda_a^k \e^{-\lambda_a}}{k!} \quad \text{for all} \quad k \geq 0 \quad \text{with} \quad \lambda_a \defeq \begin{cases}
		2 & \text{if} \quad a = 0, \\
		1 & \text{if} \quad a = 1.
	\end{cases}
\end{equation*}

At most $70\%$ of the taxis can be charged simultaneously since there are limited charging spots, and the taxi company has agreed with the city hall to deploy at least $10\%$ of the fleet at the airport. These constraints can be formulated in such a way that Assumption \ref{ass: inequality constraints} holds, i.e., by requiring that the following inequalities hold for all $t \geq 0$:
\begin{align*}
	&\sum_{i \in S} \by_n(t, i, 2) \leq 0.7, \\ 
	&\sum_{i \in S} \left[\by_n(t, i, 1) + \by_n(t, i, 2)\right] \leq 0.9.
\end{align*}

Finally, the rewards are given by
\begin{equation*}
	\begin{tabular}{ll}
		$r(i, 0) = b_0 \condp*{X_0 < i} - c_0 \condp*{X_0 \geq i}$  	& $\text{if} \quad a = 0$, \\
		$r(i, 1) = b_1 \conde*{X_1\ind{X_1 < i}} - c_1 \condp*{X_1 \geq i}$ 		& $\text{if} \quad a = 1$, \\
		$r(i, 2) = -c_2$														& $\text{if} \quad a = 2$,
	\end{tabular}
\end{equation*}
The constants $b_a$ and $c_a$ can be interpreted as revenues and costs, respectively. On the one hand, taxis deployed at the airport do one trip per period and charge a flat fare $b_0 \defeq 3$, whereas taxis deployed at the city center charge $b_1 \defeq 2.5$ per unit of distance traveled. On the other hand, a cost $c_2 \defeq 2$ is paid for charging a taxi and penalties $c_0 \defeq 3$ and $c_1 \defeq 2$ are incurred if a taxi is deployed and completely depletes its battery. The rewards defined in this way are plotted on the left panel of Figure \ref{fig: example 6}.

\begin{figure}
	\centering
	\begin{subfigure}{0.49\columnwidth}
		\centering
		\includegraphics{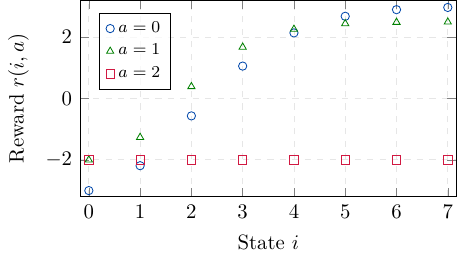}
	\end{subfigure}
	\hfill
	\begin{subfigure}{0.49\columnwidth}
		\centering
		\includegraphics{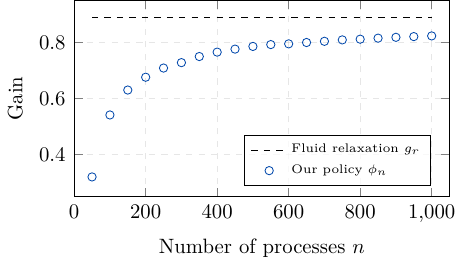}
	\end{subfigure}
	\caption{The left panel shows the reward obtained per process and per time step as a function of the state of the process and the action selected. The right panel charts the optimal value of the fluid relaxation, and the gain of the discrete control $\phi_n$ 
	when all the batteries are initially empty.}
	\label{fig: example 6}
\end{figure}

Because the constraints of the problem do not depend on the number of taxis $n$, it follows from Proposition \ref{prop: asymptotic upper bound for the gain} that the optimal value $g_r = 0.8911$ of the fluid relaxation is an upper bound for the gain of any policy. The optimal solution $y^*$ of the fluid relaxation that we used to construct a discrete control $\phi_n$ is as follows:
\begin{equation*}
	y^* =
	\begin{bmatrix}
		0 & 0 & 0 & 0 & 0 & 0 & 0 & 0.1000 \\
		0 & 0 & 0 & 0 & 0 & 0 & 0.3236 & 0.2095 \\
		0.0009 & 0.0023 & 0.0100 & 0.0343 & 0.1004 & 0.2189 & 0 & 0
	\end{bmatrix}.
\end{equation*}
The constraint that $10\%$ of the fleet is deployed at the airport is satisfied with equality, whereas the other constraint is satisfied with a strict inequality; approximately $37\%$ of the taxis are being charged while $53\%$ of the taxis are in the city center.

Although the unconstrained single-process problem is mutichain, it is easy to check that the policy defined by \eqref{eq: candidate pi} satisfies Condition \ref{cond: fluid control partially based on pi}. It follows that asymptotically optimal discrete controls $\phi_n$ can be obtained as explained in Remark \ref{rem: rounding for inequality constraints}. The right panel of Figure \ref{fig: example 6} compares the gains of these discrete controls against the upper bound $g_r$ for different values of $n$; the gains were computed by simulating the processes and averaging the rewards. The plot confirms that the gain of our discrete controls approaches the upper bound. 

\subsection{Multi-armed restless bandits}
\label{sub: restless bandits examples}

We are not aware of any policies for maximizing the expected average reward of weakly coupled processes with more than one constraint. Hence, we compare our assumptions on the transition kernels of the processes with those in the literature for multi-armed restless bandits, which are a narrower class of problems with only two actions and a single constraint with a specific structure. In particular, we consider multi-armed restless bandit problems where the budget constraint is an equality constraint, and we compare our discrete controls with the following four policies: the Whittle index policy proposed in \cite{whittle1988restless}, the LP-priority policies introduced in \cite{verloop2016asymptotically}, the Follow-The-Virtual-Advice (FTVA) policy defined in \cite{hong2023restless} and the ID policy proposed in \cite{hong2024unichain}. The asymptotic optimality of these policies has been proved in \cite{weber1990index,verloop2016asymptotically,gast2023linear,hong2023restless,gast2023exponential,hong2024unichain} under the assumptions that we discuss below.
\begin{enumerate}
	\item[(a)] For the Whittle index and LP-priority policies, the Markov decision process formed by $n$ arms must be unichain for all $n$ as required in \cite{weber1990index,verloop2016asymptotically,gast2023exponential,gast2023linear}; in fact, some of these papers even assume that each arm is ergodic.
	
	\item[(b)] For the Whittle index policy, the unconstrained single arm must satisfy a condition known as indexability, without which the policy is not well-defined; see \cite{whittle1988restless}.
	
	\item[(c)] For the Whittle index and LP-priority policies, a discrete-time dynamical system must satisfy a global attractor property similar to \eqref{eq: global attractivity}; see \cite{weber1990index,verloop2016asymptotically,gast2023exponential,gast2023linear}. We do not assume that \eqref{eq: global attractivity} holds but prove \eqref{eq: global attractivity} under the assumption of Proposition \ref{prop: multi-armed restless bandits}. 
	
	\item[(d)] For the FTVA policy, the synchronization assumption defined in \cite{hong2023restless} must hold.
	
	\item[(e)] For the ID policy, the single-arm policy \eqref{eq: candidate pi} must be unichain and aperiodic.
\end{enumerate}

The synchronization assumption pertains to a single-arm problem with a relaxed budget constraint, which must hold on average and not at each time step. Loosely speaking, a system formed by a leader and a follower arm is considered, where the follower arm always selects the same action as the leader arm. The assumption is that there exists an optimal policy, for the aforementioned relaxed single-arm problem, that is unichain and such that the mean time until the leader and follower arms are in the same state is finite. We note that an optimal policy for the above-described relaxed single-arm problem can be obtained as in \eqref{eq: candidate pi} by standard results for constrained Markov decision processes; see \cite{altman2021constrained}.

Conditions (a)-(c) are significantly stronger than the assumption of Proposition \ref{prop: multi-armed restless bandits}, which only requires that the unconstrained single-arm problem admits a unichain and aperiodic policy such that the unique irreducible class contains $S_+^*$. In addition, condition (e) implies that Proposition \ref{prop: multi-armed restless bandits} holds since $x^*$ is the stationary distribution of the policy defined by \eqref{eq: candidate pi} when this policy is unichain and aperiodic. On the other hand, condition (d) is not entirely understood. However, the synchronization assumption seems to be connected with the existence of suitable aperiodic policies for a single arm; see \cite[Appendix C]{hong2023restless}.

Each of the conditions (a)-(e) is violated in one of the examples presented below, and each of the policies for multi-armed bandit problems mentioned above is not asymptotically optimal in one of the examples. In contrast, the assumption of Proposition \ref{prop: multi-armed restless bandits} holds in all the examples, and thus our discrete controls are provably asymptotically optimal.

\subsubsection{Counterexamples to Whittle index and LP-priority policies}
\label{subsub: counterexamples for whittle index and lp-priority policies}

In this section we consider two examples where the number of states is three and the single-arm problem is ergodic. Further, the transition probability matrices have positive entries; these matrices, the rewards and the constraints are provided in Appendix \ref{app: details about the examples}.

The first example was first considered in \cite{gast2023testing} and is nonindexable. In particular, the Whittle index policy is not well-defined. The left panel of Figure \ref{fig: examples 2 and 3} compares the gain of different policies with the optimal value $g_r = 0.3437$ of the fluid relaxation, which is an upper bound for the gain of any policy. For this particular problem instance, the LP-priority policy outperforms our policy $\phi_n$ and the FTVA and ID policies; the latter three policies perform similarly. However, all four policies perform extremely well, with optimality gaps that are less than $3\%$ for $n = 200$ and $1\%$ for $n = 2000$.

\begin{figure}
	\centering
	\begin{subfigure}{0.49\columnwidth}
		\centering
		\includegraphics{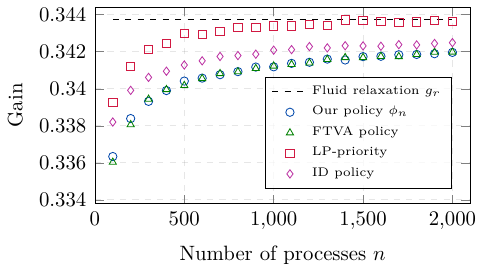}
	\end{subfigure}
	\hfill
	\begin{subfigure}{0.49\columnwidth}
		\centering
		\includegraphics{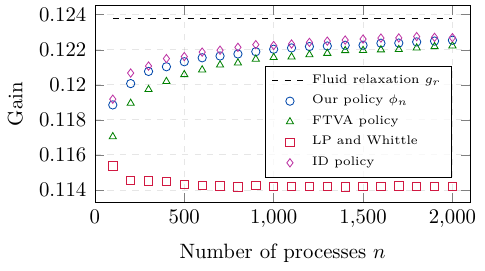}
	\end{subfigure}
	\caption{The left panel and right panels correspond to the nonindexable problem and the problem where the global attractivity property does not hold for the Whittle index and LP-priority policies, respectively.}
	\label{fig: examples 2 and 3}
\end{figure}

The second example was first considered in \cite{gast2020exponential} and is such that the global attractivity assumption does not hold for the Whittle index and LP-priority policies, which in this case are the same policy. The plot in the right panel of Figure \ref{fig: examples 2 and 3} shows that these policies are not asymptotically optimal for this problem instance. In contrast, our discrete control and the FTVA and ID policies approach $g_r = 0.1238$ at similar rates as $n$ grows.

\subsubsection{Counterexamples to Follow-The-Virtual-Advice}
\label{subsub: counterexample for ftva}

We now consider an example where the synchronization assumption does not hold and the single-arm problem is multichain. Specifically, the budget constraint is $d = 0.25$, the transition probabilities are depicted in Figure \ref{fig: example 5} and the rewards are given by
\begin{equation*}
	r =
	\begin{bmatrix}
		0 & 0 & 1 & 0 & 1 \\
		0 & 0 & 0 & 1 & 0
	\end{bmatrix}.
\end{equation*}

In order to show that the synchronization assumption does not hold, let us consider the single-arm problem with the relaxed budget constraint, i.e., action $1$ must be selected $25\%$ of the time. Consider also the following policy for this problem:
\begin{equation*}
	\condr{\pi}*{0 | i} = 1 \quad \text{if} \quad i \in \{2, 4\} \quad \text{and} \quad \condr{\pi}*{1 | i} = 1 \quad \text{if} \quad i \in \{0, 1, 3\}.
\end{equation*}
This policy is unichain and aperiodic, and the probability that action $1$ is selected in steady state is $0.25$. Hence, $\pi$ is a feasible solution, and it is easy to check that its gain is $1$, which is clearly the maximum gain that any policy can achieve. It is also possible to check that any other policy with gain $1$ must coincide with $\pi$ in states $2$, $3$ and $4$. Moreover, the set $\{2, 3, 4\}$ must be absorbing for the Markov chain associated with the policy.

Suppose now that a leader arm uses one of the optimal policies while a follower arm always selects the same action as the leader arm. Assume also that the leader and follower arms are in states $2$ and $0$ at time zero, respectively. Since the leader arm remains within $\{2, 3, 4\}$, it can never select action $1$ two times in a row and therefore the follower arm can never leave $\{0, 1\}$. It follows that the synchronization assumption does not hold, and the plot in Figure \ref{fig: example 5} shows that FTVA is not asymptotically optimal if all the arms are initially in state $0$. Similar counterexamples to FTVA are constructed in \cite{hong2024unichain}; these examples satisfy the assumptions in \cite{hong2024unichain} and thus also our weaker assumptions.

\begin{figure}
	\centering
	\begin{subfigure}{0.49\columnwidth}
		\centering
		\includegraphics{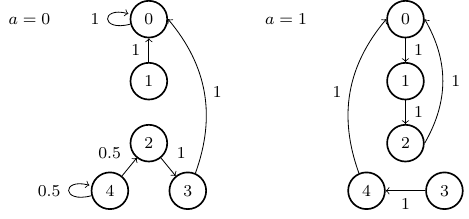}
	\end{subfigure}
	\hfill
	\begin{subfigure}{0.49\columnwidth}
		\centering
		\includegraphics{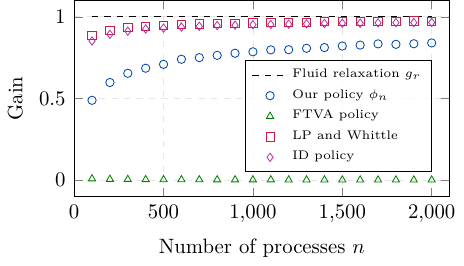}
	\end{subfigure}
	\caption{The diagrams depict the transition probabilities for both actions. The plot compares the optimal value of the fluid relaxation and the gain of several policies. The gains were computed by simulating systems where all the arms are initially in state $0$.}
	\label{fig: example 5}
\end{figure}

States $2$ and $3$ have the lowest and highest Whittle indexes, respectively, and our numeric computations indicate that the other three states have a common Whittle index that is strictly between the indexes of $2$ and $3$. If we arrange the states as $(3, 1, 0, 4, 2)$, then we obtain the Whittle index policy plotted in Figure \ref{fig: example 5}, which is also an LP-priority policy; arranging the states as $(3, 1, 0, 2, 4)$ we obtain another LP-priority policy with nearly the same performance. However, there are no results that rigorously establish this asymptotic optimality because the example considered here is multichain and all the available results concern unichain setups. In contrast, our discrete controls $\phi_n$ and the ID policy are provably asymptotically optimal for this problem instance, as confirmed by Figure \ref{fig: example 5}.

\subsubsection{Counterexample to ID policy}
\label{subsub: counterexample to id policy}

Finally, we consider multi-armed restless bandits with budget constraint $d = 0.5$ and the transition probabilities and rewards depicted in Figure \ref{fig: example 8}. The objective of the fluid relaxation \eqref{pr: fluid relaxation} is to maximize $y(0, 0) + y(1, 1) \leq 1$, and $\max\{y(0, 0), y(1, 1)\} \leq 0.5$ by the budget constraint. Then it is not difficult to check that there exists a unique solution $y^*$ which is given by $y^*(0, 0) = y^*(1, 1) = 0.5$ and $y^*(i, a) = 0$ otherwise.

\begin{figure}
	\centering
	\begin{subfigure}{0.49\columnwidth}
		\centering
		\includegraphics{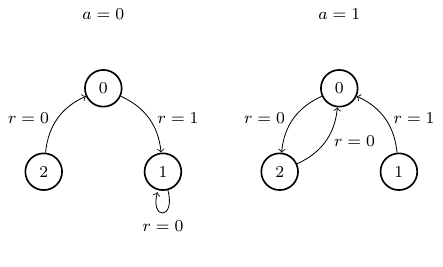}
	\end{subfigure}
	\hfill
	\begin{subfigure}{0.49\columnwidth}
		\centering
		\includegraphics{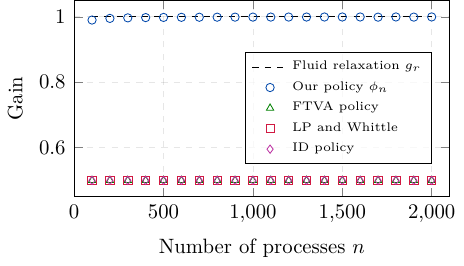}
	\end{subfigure}
	\caption{The diagrams depict the transition probabilities and rewards for both actions; all the transitions are deterministic. The chart compares the optimal value of the fluid relaxation and the gain of different policies; the gains were computed by simulating systems where all the arms are initially in state $0$.}
	\label{fig: example 8}
\end{figure}

Consider now the single-arm policy $\mu$ defined by \eqref{eq: candidate pi} when $y^*$ is as above. State $2$ is transient while states $0$ and $1$ form an irreducible class with period $2$. The synchronization assumption used in \cite{hong2023restless} does not hold because a leader starting in state $0$ and a follower starting in state $2$ will never be in the same state. In addition, the assumptions used in \cite{hong2024unichain} to establish the asymptotic optimality of the ID policy do not hold either since $\mu$ is not aperiodic. However, the policy $\nu$ that selects actions uniformly at random is ergodic and thus yields discrete controls $\phi_n$ that are asymptotically optimal by Proposition \ref{prop: multi-armed restless bandits}.

The plot in Figure \ref{fig: example 8} confirms that our discrete controls $\phi_n$ are asymptotically optimal and further shows that all the other policies have gains that are far from the optimal value when all the arms are in state $0$ at time zero. In fact, it is not difficult to check analytically that the Whittle index, LP-priority and ID policies have gains equal to $1 / 2$.

For example, consider the ID policy, which follows the policy $\mu$ defined by \eqref{eq: candidate pi} for as many arms as possible. More specifically, a permanent and unique ID in $\{1, \dots, n\}$ is assigned to each arm at time $t = 0$. Then at each time, an action is sampled for each arm using the single-arm policy $\mu$, and the sampled action is applied to the arms in order of their IDs for as many arms as allowed by the budget constraint; the actions for the rest of the arms are chosen so as to comply with the constraint. For the problem described in Figure \ref{fig: example 8}, the action sampled at time $t = 0$ is $0$ for all the arms. Hence, the half of the arms having the lowest IDs will move to state $1$ at $t = 1$, while the rest of the arms will move to state $2$. Then the action sampled for the arms in state $1$ will be $1$ because $\mu$ is deterministic at state $1$. Since these arms have the lowest IDs, action $1$ will be used for these arms and action $0$ will be used for the other arms. It follows that the state of the entire system at $t = 2$ is exactly the same as at $t = 0$ and that the system evolves periodically in the way just described. We conclude that half of the arms will obtain a reward equal to $1$ at each time step while the other half will not obtain any reward.

Noting that the Whittle index and LP-priority policy give the highest priority for using action $1$ to state $1$, similar arguments as above can be used to show that both policies have gain $1 / 2$. The gain of the FTVA policy may depend on the tie-breaking rule used to decide which arms follow the virtual advice; the tie-breaking rule used in our simulations is the same that we described above for the ID policy. Nonetheless, we recall that the synchronization assumption does not hold for this problem instance.

\section{Concluding remarks}
\label{sec: conclusion}

We have established that asymptotically optimal policies for weakly coupled Markov decision processes can be obtained essentially by solving a fluid counterpart of the problem and then constructing discrete approximations of the solution. We have further derived sufficient conditions for the existence of a fluid control that solves the fluid problem and structural properties that help to construct such a fluid control. Moreover, we have used these conditions and structural properties to construct asymptotically optimal policies for multi-armed restless bandits and a broad class of problems with multiple actions and inequality constraints. In these setups, we showed that a sufficient condition for asymptotic optimality is that the unconstrained single-process problem admits a suitable unichain and aperiodic policy; and we provided multichain examples where this condition holds.

While most of our results hold for problems with general constraints, the asymptotically optimal policies are constructed explicitly only for multi-armed restless bandits and the above-mentioned class of problems with multiple actions and inequality constraints. An interesting topic for future research is the explicit construction of asymptotically optimal policies for problems with a different structure. In addition, the present paper focuses on obtaining asymptotically optimal policies under very general assumptions, but it would be interesting to study how the optimality gap of our policies depends on the number of processes, possibly under more restrictive assumptions. Our numerical experiments suggest that this dependence may vary across different problem instances.


\begin{appendices}
	
\section{Additional proofs}
\label{app: proofs of various results}

\begin{proof}[Proof of Proposition \ref{prop: asymptotic upper bound for the gain}]
	Let $\pi \in \Pi_n^{SR}$ and $s \in S^n$. By \cite[Proposition 8.1.1]{puterman2014markov},
	\begin{equation}
		\label{aux: state-action frequencies for a policy and initial state}
		y_n(i, a) \defeq \lim_{T \to \infty} \frac{1}{T}\sum_{t = 0}^{T - 1} \conds{E_s^\pi}*{\by_n(t, i, a)}
	\end{equation}
	exists for all $i \in S$ and $a \in A$; the limit clearly depends on the specific policy $\pi$ and initial state $s$, but we omit them from the notation for brevity. We claim that
	\begin{subequations}
		\label{aux: feasibility conditions}
		\begin{align}
			&\sum_{a \in A} y_n(a)P(a) = \sum_{a \in A} y_n(a), \label{saux: balance equation}\\
			&\sum_{a \in A} y_n(a)C_n(a) = d_n, \\
			&\sum_{a \in A} y_n(a)E_n(a) \preceq f_n, \\
			&y_n \in Y.
		\end{align}
	\end{subequations}
	
	Suppose that \eqref{aux: feasibility conditions} holds for all the stationary policies $\pi$ and initial states $s$. If \eqref{eq: constant constraints} also holds, then $y_n$ is a feasible solution of the fluid relaxation, and thus
	\begin{equation*}
		g_n^\pi(s) = \sum_{a \in A} y_n(a) r(a) \leq g_r.
	\end{equation*}
	Since the stationary policy $\pi$ and initial state $s$ are arbitrary, this proves the first claim of the proposition. For the second claim, we argue by contradiction. If \eqref{eq: asymptotic upper bound} does not hold, then there exist stationary policies and initial states $\set{\pi_k, s_k}{k \in \calK}$ such that
	\begin{equation*}
		\lim_{k \to \infty} g_k^{\pi_k}(s_k) > g_r,
	\end{equation*}
	where $\calK \subset \N$ is infinite. Define $y_k$ in terms of $\pi_k$ and $s_k$ as in \eqref{aux: state-action frequencies for a policy and initial state}. It follows from \eqref{aux: feasibility conditions} and the compacity of $Y$ that the sequence $\set{y_k}{k \in \calK}$ has a limit point $y$ that is a feasible solution of \eqref{pr: fluid relaxation}. By taking a subsequence, we may assume without any loss of generality that $y_k \to y$ as $k \to \infty$, which leads to the following contradiction:
	\begin{equation*}
		\lim_{k \to \infty} g_k^{\pi_k}(s_k) = \lim_{k \to \infty} \sum_{a \in A} y_k(a)r(a) = \sum_{a \in A} y(a)r(a) \leq g_r.
	\end{equation*}
	
	It only remains to prove that \eqref{aux: feasibility conditions} holds for all $\pi \in \Pi_n^{SR}$ and $s \in S^n$. We only prove \eqref{saux: balance equation} since the other properties follow from similar arguments. For this purpose note that
	\begin{equation*}
		\sum_{a \in A} \conds{E_s^\pi}*{\by_n(t, a)} P(a) = \sum_{a \in A} \conds{E_s^\pi}*{\by_n(t + 1, a)} \quad \text{for all} \quad t \geq 0
	\end{equation*}
	by \eqref{eq: state and state-action frequencies}. Taking the Ces\`aro limit as $t \to \infty$ on both sides of the equation, we conclude that the vector $y_n$ defined by \eqref{aux: state-action frequencies for a policy and initial state} indeed satisfies \eqref{saux: balance equation}.
\end{proof}

\begin{proof}[Proof of Lemma \ref{lem: variance of z_n}]
	Consider the random variables
	\begin{equation*}
		\condr{U_t^l}*{j | i, a} \defeq \condr{B_t^l}*{j | i , a} - \condr{p}*{j | i, a} \quad \text{and} \quad \condr{V_t^l}*{i, a} \defeq \ind{n\by_n(t - 1, i, a) \geq l}.
	\end{equation*}
	Now let us fix $t \geq 1$ and $j \in S$, and note that we may write
	\begin{equation*}
		\bz_n(t, j) = \frac{1}{n} \sum_{l = 1}^n \sum_{i \in S} \sum_{a \in A} \condr{W_t^l}*{j | i, a} \quad \text{with} \quad \condr{W_t^l}*{j | i, a} \defeq \condr{U_t^l}*{j | i, a}\condr{V_t^l}*{i, a}.
	\end{equation*}
	
	If $l_1 \neq l_2$, $i_1 \neq i_2$ or $a_1 \neq a_2$, then
	\begin{equation*}
		\conde*{\condr{W_t^{l_1}}*{j | i_1, a_1}\condr{W_t^{l_2}}*{j | i_2, a_2}} = \conde*{\condr{U_t^{l_1}}*{j | i_1, a_1}}\conde*{\condr{V_t^{l_1}}*{i_1, a_1}\condr{W_t^{l_2}}*{j | i_2, a_2}} = 0
	\end{equation*}
	because $\condr{U_t^{l_1}}*{j | i_1, a_1}$ is independent of $\by_n(t - 1)$ and $\condr{U_t^{l_2}}*{j | i_2, a_2}$. Furthermore,
	\begin{align*}
		\conde*{\left(\condr{W_t^l}*{j | i, a}\right)^2} &= \conde*{\left(\condr{U_t^l}*{j | i, a}\right)^2} \conde*{\left(\condr{V_t^l}*{i, a}\right)^2} \\
		&= \left[1 - \condr{p}*{j | i, a}\right]\condr{p}*{j | i, a}\condp{n\by_n(t - 1, i, a) \geq l}
	\end{align*}
	for all $l \geq 1$, $i \in S$ and $a \in A$. We conclude that
	\begin{align*}
		\conde*{\bz_n^2(t, j)} &= \frac{1}{n^2}\sum_{l = 1}^n \sum_{i \in S} \sum_{a \in A} \left[1 - \condr{p}*{j | i, a}\right]\condr{p}*{j | i, a}\condp{n\by_n(t - 1, i, a) \geq l} \\
		&\leq \frac{1}{n^2} \sum_{i \in S} \sum_{a \in A} q_{\max}(j)\condr{p}*{j | i, a}\conde*{n\by_n(t - 1, i, a)} = \frac{q_{\max}(j)}{n}\conde*{\bx_n(t, j)}. 
	\end{align*}
	
	The other inequality is obtained using similar arguments.
\end{proof}

\begin{proof}[Proof of Lemma \ref{lem: function beta}]
	Suppose that $x \in X$ and $x \neq x^*$. Then
	\begin{equation*}
		\beta(x) = \beta(x) x^* e \leq x e = 1,
	\end{equation*}
	where $e$ is the column vector of ones. It is clear that $\beta(x) \neq 1$ since otherwise $x = x^*$, so we conclude that $1 - \beta(x) > 0$, and thus (a) follows from
	\begin{equation*}
		\left[1 - \beta(x)\right]^{-1}\left[x - \beta(x)x^*\right] e = 1.
	\end{equation*}
	
	Recall that $S_+^* \defeq \set{i \in S}{x^*(i) > 0}$. For each $\lambda \geq 0$, we have:
	\begin{align*}
		&\beta^{-1}\left((\lambda, \infty)\right) = \set{x \in \R_+^S}{x(i) > \lambda x^*(i)\ \text{for all}\ i \in S^*}, \\
		&\beta^{-1}\left([\lambda, \infty)\right) = \set{x \in \R_+^S}{x(i) \geq \lambda x^*(i)\ \text{for all}\ i \in S^*}.
	\end{align*}
	These sets are open and closed in $\R_+^S$, respectively. Therefore,
	\begin{equation*}
		\beta^{-1}\left((\lambda, \mu)\right) = \beta^{-1}\left((\lambda, \infty)\right) \setminus \beta^{-1}\left([\mu, \infty)\right) \quad \text{and} \quad \beta^{-1}\left([0, \mu)\right) = \R_+^S \setminus \beta^{-1}\left([\mu, \infty)\right)
	\end{equation*}
	are open in $\R_+^S$ for all $\lambda \geq 0$ and $\mu > 0$, which proves (b).
\end{proof}

\begin{proof}[Proof of Proposition \ref{prop: problems with inequality constraints}]
	We only need to check that $\psi$ is a fluid control. First note that
	\begin{equation*}
		\sum_{a \in A} \psi(x, a) = \gamma x \sum_{a \in A} D_\pi(a) + (1 - \gamma) x = x \quad \text{for all} \quad x \in X.
	\end{equation*}
	Furthermore, the inequality constraints hold since
	\begin{equation*}
		\sum_{a \neq 0}\sum_{i \in S} \gamma x(i)\condr{\pi}*{a | i}E(i, k, a) \leq \gamma \max \set{E(i, k, a)}{i \in S, a \neq 0} \leq f(k) \quad \text{for all} \quad x \in X
	\end{equation*}
	and all $k$ smaller than or equal to the dimension of $f$. Hence, Definition \ref{def: fluid control} holds.
\end{proof}

\begin{proof}[Proof of Proposition \ref{prop: multi-armed restless bandits}]
	Since $d < 1$, we have
	\begin{equation*}
		\varepsilon \defeq \min_{i \in S} \left[1 - d \condr{\pi}*{1 | i}\right] > 0.
	\end{equation*}
	Hence, the denominator in the definition of $\psi_2(x)(i, 1)$ is lower bounded by $(1 - d)\varepsilon$. This implies that $\psi_2$ is well-defined, and clearly also continuous.
	
	It only remains to check that $\psi$ is a fluid control. First note that
	\begin{equation*}
		\sum_{a \in A} \psi(x, a) = d x \sum_{a \in A} D_\pi(a) + (1 - d) \sum_{a \in A} \psi_2(x, a) = d x + (1 - d)x = x
	\end{equation*} 
	because $\psi_2(x, 0) + \psi_2(x, 1) = x$ by definition. Moreover,
	\begin{align*}
		\sum_{i \in S} \psi(x)(i, 1) &= \sum_{i \in S} x(i) d \condr{\pi}*{1 | i} + \sum_{i \in S} (1 - d)\psi_2(x)(i, 1) \\
		&= \sum_{i \in S} x(i) d \condr{\pi}*{1 | i} + \sum_{i \in S} \frac{d - \sum_{j \in S} x(j) d \condr{\pi}*{1 | j}}{\sum_{j \in S} x(j)\left[1 - d \condr{\pi}*{1 | j}\right]}x(i)\left[1 - d \condr{\pi}*{1 | i}\right] = d,
	\end{align*}
	which means that $\psi$ complies with the unique equality constraint.
\end{proof}

\section{Details about the examples}
\label{app: details about the examples}

The transition probability matrices for the nonindexable example of Section \ref{subsub: counterexamples for whittle index and lp-priority policies} are:
\begin{equation*}
	P(0) =
	\begin{bmatrix}
		0.0050 & 0.7930 & 0.2020 \\
		0.0270 & 0.5580 & 0.4150 \\
		0.7360 & 0.2490 & 0.0150
	\end{bmatrix} \quad \text{and} \quad
	P(1) =
	\begin{bmatrix}
		0.7180 & 0.2540 & 0.0280 \\
		0.3470 & 0.0970 & 0.5560 \\
		0.0150 & 0.9560 & 0.0290
	\end{bmatrix}.
\end{equation*}
In addition, the budget constraint is $d = 0.5$ and the rewards are given by
\begin{equation*}
	r =
	\begin{bmatrix}
		0 & 0 & 0 \\
		0.6990 & 0.3620 & 0.7150
	\end{bmatrix}.
\end{equation*}

The transition probability matrices for other example of Section \ref{subsub: counterexamples for whittle index and lp-priority policies} are:
\begin{equation*}
	P(0) =
	\begin{bmatrix}
		0.0223 & 0.1023 & 0.8754 \\
		0.0343 & 0.1718 & 0.7940 \\
		0.5232 & 0.4552 & 0.0215
	\end{bmatrix} \quad \text{and} \quad
	P(1) =
	\begin{bmatrix}
		0.1487 & 0.3044 & 0.5469 \\
		0.5685 & 0.4112 & 0.0204 \\
		0.2527 & 0.2731 & 0.4742
	\end{bmatrix}.
\end{equation*}
Also, the budget constraint is $d = 0.4$ and the rewards are given by
\begin{equation*}
	r =
	\begin{bmatrix}
		0 & 0 & 0 \\
		0.3740 & 0.1174 & 0.0787
	\end{bmatrix}.
\end{equation*}

\end{appendices}
	
\newcommand{\noop}[1]{}
\bibliographystyle{IEEEtranS}
\bibliography{mdp,phd}
	
\end{document}